\documentclass[11pt,a4paper,reqno]{amsart}

\usepackage{amsmath}
\usepackage{amsfonts}
\usepackage{amssymb}
\usepackage{amsthm}
\usepackage{enumerate}
\usepackage{enumitem}
\usepackage{comment}
\usepackage{mathtools}
\usepackage{bbding}

\usepackage{todonotes}
\usepackage{graphics}
\usepackage{graphicx}
\usepackage{tikz,tikz-cd}
\usepackage[hang,flushmargin]{footmisc}
\usepackage[pdfencoding=auto, hyperfootnotes=false]{hyperref}
\usepackage{xcolor}
\usepackage{url}
\usepackage[hang,flushmargin]{footmisc}
\usepackage{esvect}

\usepackage[utf8]{inputenc}
\numberwithin{equation}{section}

\newcommand{\R}{{\mathbb R}}

\newcommand{\Z}{{\mathbb Z}}
\newcommand{\Q}{{\mathbb Q}}
\newcommand{\T}{{\mathbb T}}

\newcommand{\mc}{\mathcal}

\newcommand{\mb}{\mathbb}

\newcommand{\wh}{\widehat}
\newcommand{\wt}{\widetilde}

\DeclareMathOperator{\ab}{G}
\DeclareMathOperator{\md}{Md}
\DeclareMathOperator{\Span}{Span}
\newtheorem{theorem}{Theorem}[section]
\newtheorem{lemma}[theorem]{Lemma}

\newtheorem{corollary}[theorem]{Corollary}
\newtheorem{proposition}[theorem]{Proposition}
\newtheorem{question}[theorem]{Question}
\newtheorem{problem}[theorem]{Problem}

\theoremstyle{definition}
\newtheorem{defn}[theorem]{Definition}
\newtheorem{remark}[theorem]{Remark}

\begin{document}

\pagestyle{myheadings}
\markboth{}{\rm\textsf{version: 2022/7/24}\hfill\qquad}

\vspace*{-1.5cm}
\title[On sets with missing differences in compact abelian groups]{On sets with missing differences\\ in compact abelian groups}
\author{Pablo Candela}
\address{Instituto de Ciencias Matem\'aticas\\ 
Calle Nicol\'as Cabrera 13-15\\ Madrid 28049\\ Spain}
\email{pablo.candela@icmat.es}
\author{Fernando Chamizo}
\address{Autonomous University of Madrid, and ICMAT\\Ciudad Universitaria de Cantoblanco\\	Madrid 28049\\	Spain}
\email{fernando.chamizo@uam.es}
\author{Antonio Córdoba}
\address{Autonomous University of Madrid, and ICMAT\\Ciudad Universitaria de Cantoblanco\\	Madrid 28049\\	Spain}
\email{antonio.cordoba@uam.es}
\date{}
\begin{abstract}
A much-studied problem posed by Motzkin asks to determine, given a finite set $D$ of integers, the so-called \emph{Motzkin density} for $D$, i.e., the supremum of upper densities of sets of integers whose difference set avoids $D$. We study the natural analogue of this problem in compact abelian groups. Using ergodic-theoretic tools, this is shown to be equivalent to the following discrete problem: given a lattice $\Lambda\subset \mb{Z}^r$, letting $D$ be the image in $\mb{Z}^r/\Lambda$ of the standard basis, determine the Motzkin density for $D$ in $\mb{Z}^r/\Lambda$. We study in particular the periodicity question: is there a periodic $D$-avoiding set of maximal density in $\mb{Z}^r/\Lambda$? The Greenfeld--Tao counterexample to the periodic tiling conjecture implies that the answer can be negative. On the other hand, we prove that the answer is positive in several cases, including the case rank$(\Lambda)=1$ (in which we give a formula for the Motzkin density), the case rank$(\Lambda)=r-1$, and hence also the case $r\leq 3$. It follows that, for up to three missing differences, the Motzkin density in a compact abelian group is always a rational number.
\end{abstract}
\maketitle

\section{Introduction}

Given an abelian group $\ab$ and a set $D\subset \ab$, we say that a set $A\subset \ab$ is \emph{$D$-avoiding} if the difference set $A-A$ and $D$ are disjoint. It is a well-known fundamental problem in combinatorial number theory to determine, given a finite set $D$ of positive integers, how large a $D$-avoiding subset of $\mb{Z}$ can be. More precisely, the problem consists in estimating the following quantity:
\begin{equation}\label{eq:Md-def}
\sup\,\big\{ \limsup_{N\to\infty} \tfrac{|A\cap [-N,N]|}{2N+1}: A \textrm{ is a $D$-avoiding subset of } \mb{Z}\big\}.
\end{equation}
This question originated more than half a century ago in an unpublished collection of problems posed by T.S.\ Motzkin (see \cite{C&G}).\footnote{The original problem was formulated for subsets of $\mb{N}$, and infinite sets $D$ can also be considered (see \cite{C&G}).} There have been numerous works on this question in the decades since then, including the papers \cite{Gupta,Hara,L&R,Pandey1,Pandey2,Pandey3} addressing various special cases. For a general set $D$, Motzkin's problem is still open, and even the case of three missing differences (i.e.\ with $|D|=3$) is not yet completely solved (relatively recent references on this include \cite{L&R,P&S}).

In this paper we study an analogue of Motzkin's problem in compact abelian groups. Letting $\ab$ be such a group, with Haar probability measure $\mu$, the \emph{Motzkin density} for $D\subset \ab$ (or \emph{measure of intersectivity} of $D$, as per \cite{RuzsaIntersect}) is the supremum of measures of $D$-avoiding Borel sets $A\subset \ab$. We denote this Motzkin density by $\md_{\ab}(D)$. Thus, 
\begin{equation}\label{eq:MD}
 \md_{\ab}(D)
 =
 \sup\big\{
 \mu(A)\,:\, \text{ Borel }A\subset \ab\text{ with }(A-A)\cap D=\emptyset
 \big\}.
\end{equation}
In this paper we consider only finite sets\footnote{Note that $A$ is $D$-avoiding if and only if it is $D\cup (-D)$-avoiding. When we specify the cardinality of $D$, we always tacitly assume that each $t\in D$ is the only element in $D$ from the pair $\{t,-t\}$.}  $D$. The \emph{Motzkin problem} in compact abelian groups can then be stated as follows.
\begin{problem}\label{prob:CAGMotzkin}
Given a compact abelian group $\ab$ and a finite set $D\subset \ab\setminus\{0\}$, determine or estimate the quantity $\md_{\ab}(D)$.
\end{problem}
This problem was studied in \cite{CCRS} focusing on the case of the circle group $\ab=\mb{T}:=\mb{R}/\mb{Z}$. 
This simple setting already brings out interesting connections with tools from graph theory, ergodic theory, the geometry of numbers and Fourier analysis. Let us illustrate this by mentioning some basic results that can be proved using such tools. For any finite set $D$ of numbers in $(0,1)$ such that $D\cup\{1\}$ is linearly independent over $\mb{Q}$, we have $\md_{\mb{T}}(D)=1/2$ (proved using Rokhlin's lemma in \cite[Theorem 2.4]{CCRS}, see also Corollary \ref{cor:KronFree} in this paper); on the other hand, this supremum $\md_{\mb{T}}(D)=1/2$ is not attained, and in fact, for any Borel set $A\subset\mb{T}$ of measure $1/2$, the difference set $A-A$ modulo 1 contains every irrational number in $(0,1)$ (this follows from ergodicity of irrational rotations, and can also be proved using Fourier analysis). More involved results include exact (or asymptotically sharp) formulae for $\md_{\mb{T}}(D)$ for $|D|\leq 2$ (e.g.\ \cite[Theorems 1.2 and 1.3]{CCRS}).

Concerning the more general Problem \ref{prob:CAGMotzkin}, one of the basic open questions is whether the Motzkin density $\md_{\ab}(D)$ is always a rational number. One of the main results in this paper is the following.
\begin{theorem}\label{thm:3rational}
For any compact abelian group $\ab$ and any set $D\subset \ab\setminus\{0\}$ with $|D|\leq 3$, we have $\md_{\ab}(D)\in \mb{Q}$.
\end{theorem}
This theorem is obtained as a combination of several results in this paper, including the following theorem which provides an explicit formula for a special case.

\begin{theorem}\label{thm:rk1-intro}
Let $\ab$ be a compact abelian group, let $D=\{t_1,\ldots,t_r\}\subset \ab\setminus\{0\}$, suppose that the lattice $\{n\in \mb{Z}^r:n_1t_1+\cdots +n_rt_r=0\}$ has rank 1, and let $m\in \mb{Z}^r$ be a generator of this lattice. Then\footnote{Here $\|m\|_{\ell^1}$ is the $\ell^1$-norm of $m=(m_1,\ldots,m_r)$, i.e.\ $\|m\|_{\ell^1}=\sum_{i=1}^r |m_i|$.} $\md_{\ab}(D)=\lfloor\|m\|_{\ell^1}/2\rfloor / \|m\|_{\ell^1}$.
\end{theorem}
To detail these results further, we require additional terminology and background, so let us defer to the next section the formal statements as well as an outline of the paper.

\bigskip

\noindent \textbf{Acknowledgements.} 
All authors received funding from project PID2020-113350GB-I00 financed by MICIU/AEI/10.13039/501100011033/FEDER, EU. We thank Mihalis Kolountzakis, M\'at\'e Matolcsi, and Arsenii Sagdeev for useful comments.

\section{Initial remarks on Problem \ref{prob:CAGMotzkin}, and statements of main results}

Throughout the paper, we will denote by $\ab$ a compact (Hausdorff) abelian group. Given a finite set $D=\{t_1,\ldots,t_r\}\subset \ab\setminus\{0\}$, we define the lattice
\begin{equation}\label{eq:Lambda}
\Lambda=\Lambda_{\ab,D}:=\{n\in\mb{Z}^r: n_1t_1+\cdots+n_rt_r=0\}.
\end{equation}
In \cite{CCRS}, a generalization of Rokhlin's lemma (from ergodic theory) was used to translate Problem \ref{prob:CAGMotzkin} into the problem of estimating a certain Motzkin density in the discrete setting of the group $\mb{Z}^r/\Lambda$, a setting in which the problem can be easier to solve; see \cite[Theorem 2.6]{CCRS}. We shall revisit this result below, and present a slightly refined version  establishing an \emph{equivalence} of these two problems. To do so, we need to recall the notion of Motzkin density that we use in the discrete setting of a finitely-generated abelian group (formulated in \cite[Definition 2.5]{CCRS}). This involves F\o lner sequences.

For any set $X$ we denote by $\mc{P}_{<\infty}(X)$ the set of all finite subsets of $X$. Given a finitely generated abelian group $\Gamma$, a sequence $(F_N)_{N\in \mb{N}}$ of sets in  $\mc{P}_{<\infty}(\Gamma)$ is a \emph{F\o lner sequence} if
\begin{equation}\label{eq:FC}
 \textrm{for every $g\in \Gamma$ we have }\, \lim_{N\to\infty} \frac{|F_N\Delta (g+F_N)|}{|F_N|}=0,
\end{equation}
where $\Delta$ denotes symmetric difference.

\begin{defn}\label{def:FGMD1}
Let $\Gamma$ be a finitely-generated abelian group, let $E\subset \Gamma$ and let
\begin{equation}\label{eq:subadfn}
\phi_E:\mc{P}_{<\infty}(\Gamma)\to \mb{Z}_{\geq 0},\;\; S\,\mapsto\, \max\{\, |A|:A \subset S,\; (A-A)\cap E=\emptyset\,\}.
\end{equation}
Then we define
\begin{equation}\label{eq:FGMD1}
\md_\Gamma(E):= \lim_{N\to\infty} \frac{\phi_E(F_N)}{|F_N|},\;\textrm{ for any F\o lner sequence $(F_N)_{N\in \mb{N}}$ in $\Gamma$}.
\end{equation}
\end{defn}
\noindent The set function $\phi_E$ is monotone with respect to inclusion, subadditive, and translation invariant, whence the limit in \eqref{eq:FGMD1} exists and is independent of the choice of F\o lner sequence (see \cite[Theorem 6.1]{L&W} or \cite[Proposition 2.2]{D&Z}).

Recall that a set $F$ \emph{tiles} an abelian group $\Gamma$ if there is a set $C\subset \Gamma$ such that $\Gamma=\bigsqcup_{x\in C} (F+x)$, where $\bigsqcup$ denotes pairwise-disjoint union; we then say that $C$ is a \emph{tiling complement} of $F$ in $\Gamma$.

The following lemma gives an alternative expression for $\md_{\Gamma}(E)$ which is closer to the original definition in \cite{C&G}, i.e.\ to \eqref{eq:Md-def}. To state this we recall that a F\o lner sequence $(F_N)_{N\in \mb{N}}$ in $\Gamma$ is said to be a \emph{tiling F\o lner sequence} if each set $F_N$ tiles $\Gamma$. 
\begin{lemma}[See Lemma 2.8 in \cite{CCRS}]
Let $\Gamma$ be a finitely-generated abelian group, and let $\mc{F}=(F_N)_{N\in \mb{N}}$ be a tiling F\o lner sequence in $\Gamma$. Given any set $A\subset \Gamma$, define the upper density $\overline{\delta}_{\mc{F}}(A):=\limsup_{N\to\infty} |A\cap F_N|/|F_N|$. Then for every finite set $E\subset \Gamma$ we have
\begin{equation}\label{eq:FGMD2}
\md_\Gamma(E)=\sup\{\overline{\delta}_{\mc{F}}(A):A\subset\Gamma, (A-A)\cap E=\emptyset\}.
\end{equation}
\end{lemma}
The equivalence of the definitions of $\md_\Gamma(E)$ in \eqref{eq:FGMD1} and \eqref{eq:FGMD2} can be viewed as an extension (in the case of \emph{finite} sets $E$) of Ruzsa's equivalent characterizations of measures of intersectivity \cite[Theorem 1]{RuzsaIntersect}.

Tiling sets also enable the limit in \eqref{eq:FGMD1} to be expressed more precisely as an infimum. Indeed, letting $\mc{T}_\Gamma$ denote the set of finite subsets of $\Gamma$ that tile $\Gamma$, we have the following fact, which is an application of \cite[Theorem 5.9]{WeissAmenact} (see also \cite[Proposition 2.8]{D&Z}).
\begin{lemma}\label{lem:estiminf}
Let $\Gamma$ be a finitely-generated abelian group, let $\mc{F}=(F_N)_{N\in \mb{N}}$ be a tiling F\o lner sequence in $\Gamma$, and let $E$ be a finite subset of $\Gamma$. Then
\begin{equation}\label{eq:mdinf}
\md_\Gamma(E)=\inf_N \frac{\phi_E(F_N)}{|F_N|}=\inf_{F\in \mc{T}_\Gamma}  \frac{\phi_E(F)}{|F|}.
\end{equation}
\end{lemma}

\begin{remark}\label{rem:NatFolner}
Let us mention a standard example of a tiling F\o lner sequence in $\Gamma$, afforded by the fundamental theorem of finitely-generated abelian groups. Given a lattice $\Lambda$ of rank $d$ in $\mb{Z}^r$, let $L\in\mb{Z}^{d\times r}$ be a matrix such that $\Lambda=\mb{Z}^dL$ (the rows of $L$ form a basis\footnote{Recall that a \emph{basis} of a lattice $\Lambda\subset\mb{R}^n$ is a set of linearly-independent elements $v_1,\ldots,v_d$ of $\mb{R}^n$ such that $\Lambda=\mb{Z}v_1+\cdots+\mb{Z}v_d$; see \cite[Definition 1]{G&L}.} of $\Lambda$). It is then well-known that the Smith Normal Form of $L$ yields an isomorphism of the following kind (proving said fundamental theorem):
\begin{equation}\label{eq:FundThm}
\varphi:\mb{Z}^r/\Lambda \; \to \; \mb{Z}_{\alpha_1}\oplus\cdots\oplus \mb{Z}_{\alpha_d}\oplus\mb{Z}^{r-d},
\end{equation}
where $\alpha_1,\ldots,\alpha_d$ are positive integers (with $\alpha_i\mid \alpha_{i+1}$ for each $i$), and $\mb{Z}_{\alpha_i}$ denotes $\mb{Z}/ \alpha_i \mb{Z}$. This yields the following tiling F\o lner sequence in $\mb{Z}^r/\Lambda$:
\begin{equation}\label{eq:NatFolner}
\big(F_N=\varphi^{-1}(\mb{Z}_{\alpha_1}\times \cdots\times \mb{Z}_{\alpha_d}\times [-N,N]^{r-d})\big)_{N\in \mb{N}}.
\end{equation}
However, in what follows we shall not always use this choice of F\o lner sequence nor the isomorphism \eqref{eq:FundThm}, so the general expressions \eqref{eq:FGMD1} and \eqref{eq:mdinf} are worth keeping in mind.
\end{remark}

We can now present the result announced above, showing that the Motzkin problem in compact abelian groups is equivalent to a specific version of the Motzkin problem in finitely-generated abelian groups.

\begin{theorem}\label{thm:trans}
Let $\ab$ be a compact abelian group, and let $D=\{t_1,\ldots,t_r\}\subset \ab\setminus\{0\}$. Let $\Lambda$ be the lattice $\Lambda_{\ab,D}$ in \eqref{eq:Lambda}, and let $B_\Lambda$ be the image of the standard basis $\{e_1,\ldots,e_r\}\subset \mb{Z}^r$ under the canonical homomorphism $\mb{Z}^r\to \mb{Z}^r/\Lambda$. Then $\md_{\ab}(D)=\md_{\mb{Z}^r/\Lambda}(B_\Lambda)$. Conversely, for any lattice $\Lambda\subset \mb{Z}^r$, there is a compact abelian group $\ab=\ab_0\times \mb{T}$, for some finite group $\ab_0$, and a set $D\subset \ab\setminus\{0\}$ of cardinality $r$, such that $\md_{\mb{Z}^r/\Lambda}(B_\Lambda)=\md_{\ab}(D)$.
\end{theorem}
\begin{proof}
The first implication follows from \cite[Theorem 2.6]{CCRS} provided that $\ab$ satisfies the additional assumption (which was made in \cite[Theorem 2.6]{CCRS} for technical convenience) that $\ab$ is a \emph{metrizable} compact abelian group (equivalently, the Pontryagin dual of $\ab$ is countable; see \cite[Theorem 8.45]{H&Mo3}). This assumption can be removed, in the sense that the following claim holds: if $\ab$ is a general (not necessarily metrizable) compact (Hausdorff) abelian group, and $D$ is a finite subset of $\ab\setminus\{0\}$, then there exists a compact metrizable abelian group $H$ and a continuous surjective homomorphism $q:\ab\to H$, such that $\md_{\ab}(D)=\md_H(q(D))$ and $\Lambda_{H,q(D)}=\Lambda_{\ab,D}$. We leave the proof of this claim to Appendix \ref{App:reduc} so as not to lengthen the present proof unduly. This claim and the proof of \cite[Theorem 2.6]{CCRS} for the metrizable case yield the first implication in the theorem.

For the converse, given the lattice $\Lambda\subset \mb{Z}^r$, we shall find a compact abelian group $\ab=\ab_0\times \mb{T}$ and a set $D=\{t_1,\ldots,t_r\}\subset \ab\setminus\{0\}$ such that $\Lambda$ is the kernel of the homomorphism $n\in \mb{Z}^r\mapsto n_1t_1+\cdots+n_rt_r\in \ab$. If we do this, then (since $\ab$ is metrizable) by \cite[Theorem 2.6]{CCRS} we will have $\md_{\mb{Z}^r/\Lambda}(B_\Lambda)=\md_{\ab}(D)$, as required.  Let $L\in \mb{Z}^{d\times r}$ be a matrix of rank $d$ such that $\Lambda=\mb{Z}^d L$ (where $d$ is the rank of the lattice $\Lambda$, so the rows of $L$ form a basis of $\Lambda$). Let $U\in \mb{Z}^{d\times d}$, $V\in \mb{Z}^{r\times r}$ be unimodular matrices such that $L=USV$ where $S=\big(\textrm{diag}(\alpha_1,\ldots,\alpha_d) ~ | ~ 0^{d\times (r-d)}\big)\in \mb{Z}^{d\times r}$ is the Smith Normal Form of $L$. Thus $\Lambda=\mb{Z}^d SV$. Note that it suffices to find $\ab$ and $D'=\{t_1',\ldots,t_r'\}\subset \ab$ such that $\mb{Z}^d S=\{n\in\mb{Z}^r:n\cdot (t_1',\ldots,t_r')=0\}$, as then $\Lambda = \mb{Z}^d SV = \{nV\in\mb{Z}^r:(nV)\cdot (V^{-1})^T(t_1',\ldots,t_r')=0\} = \{n\in\mb{Z}^r:n\cdot (t_1,\ldots, t_r)=0\}$ for $(t_1,\ldots,t_r)=(V^{-1})^T(t_1',\ldots,t_r')$, so these elements $t_i$ would form a set $D$ with the desired property. Now let $\ab_0=\mb{Z}_{\alpha_1}\times \cdots\times \mb{Z}_{\alpha_d}$ and $\ab=\ab_0\times \mb{T}$. We shall find $t_1',\ldots,t_r'\in \ab$ such that $\mb{Z}^dS = (\mb{Z} \alpha_1,\ldots,\mb{Z}\alpha_d,0,\ldots,0)$ is exactly the kernel of the homomorphism $n\in \mb{Z}^r\mapsto n_1 t_1'+\cdots + n_rt_r'\in \ab$.  Let $\beta_1,\ldots,\beta_{r-d}\in (0,1)$ be numbers such that $1,\beta_1,\ldots,\beta_{r-d}$ are linearly independent over $\mb{Q}$. For $i\in [d]$ let $t_i'$ be the element of $\ab$ with a coordinate $1$ in the $\mb{Z}_{\alpha_i}$-component, and all other coordinates $0$, and for $i=d+j$, $j\in [r-d]$, let $t_i'$ be the element of $\ab$ with $\beta_j$ in the $\mb{T}$-component, and all other coordinates $0$. With this choice and the independence of the $\beta_i$ and $1$, it is clear that the $t_i'$ satisfy the claimed property.
\end{proof}
By Theorem \ref{thm:trans}, solving the Motzkin problem in compact abelian groups (Problem \ref{prob:CAGMotzkin}) is equivalent to solving the following.
\begin{problem}\label{prob:FGMotzkin}
Given a lattice $\Lambda$ in $\mb{Z}^r$, determine or estimate the quantity $\md_{\mb{Z}^r/\Lambda}(B_\Lambda)$.
\end{problem}

From now on we focus on Problem \ref{prob:FGMotzkin}. The scope of this problem is further illustrated by the fact that it  includes also, as a special case, the natural multidimensional generalization of Motzkin's original problem, namely, the generalization consisting in determining  $\md_{\mb{Z}^k}(D)$ for finite sets $D\subset \mb{Z}^k$, for $k>1$. Indeed, in Section \ref{sec:multidimcase} we will prove that for every positive integer $k$, the Motzkin problem in $\mb{Z}^k$ is equivalent to the special case of Problem \ref{prob:FGMotzkin} where the lattice $\Lambda$ has a basis that is primitive\footnote{A linearly independent set $\{b_1,\ldots,b_k\}\subset \mb{Z}^r$ is said to be \emph{primitive in $\mb{Z}^r$} if $\Span_{\mb{R}}(\{b_1,\ldots,b_k\})\cap \mb{Z}^r = \Span_{\mb{Z}} \{b_1,\ldots,b_k\}$ (see \cite[p.\ 21, Definition 2]{G&L}).} in $\mb{Z}^r$ (see Proposition \ref{prop:equivZkrefined}).

\begin{remark}
There are instances of the Motzkin problem in a compact abelian group $\ab$ in which the corresponding lattice $\Lambda$ has no basis that is primitive in $\mb{Z}^r$. This clarifies that the Motzkin problem in $\mb{Z}^k$ is equivalent to a \emph{proper} special case of Problem \ref{prob:FGMotzkin} (or equivalently, of the Motzkin problem in compact abelian groups, Problem \ref{prob:CAGMotzkin}). We can illustrate this with $r=2$ and $\ab=\mb{T}$, with an example of missing differences $t_1,t_2\in\ab$ such that the lattice $\Lambda\subset \mb{Z}^2$ has rank 1 and is generated by a non-primitive vector. Identifying $\mb{T}\setminus\{0\}$ with the interval $(0,1)$, let $t_1,t_2$ be irrational numbers in $(0,1)$ such that $2t_1+2t_2=1$ (e.g.\ $t_1=\sqrt{2}/4$, $t_2=1/2-t_1$).
 Then for any coprime non-zero integers $a,b$ we have $at_1+bt_2\not\in\mb{Z}$. Indeed, suppose for a contradiction that there is a non-zero integer $c$ with $at_1+bt_2=c$. Then $t_2=\frac{c-at_1}{b}$, and so (since $2t_1+2t_2=1$) we have $b=2t_1b+2t_2b=2t_1b+2c-2at_1 = 2(b-a)t_1+2c$. If $a\neq b$ then we deduce that $t_1=(b-2c)/2(b-a)$, contradicting $t_1\not\in \mb{Q}$. Hence $a$ and $b$ must be coprime and \emph{equal}, so either 
$a=b=1$, which cannot hold since that would imply $c=1/2$, or $a=b=-1$, which cannot hold either since it would imply $c=-1/2$. Hence no coprime integers $a,b$ satisfy $a t_1+bt_2 =0$ in $\mb{T}$ (in fact no distinct integers $a,b$ satisfy this).
\end{remark}

As recalled in the introduction, already in the case of the original Motzkin problem in $\mb{Z}$, determining the exact value of $\md_{\mb{Z}}(D)$ for finite sets $D$ is hard and still open. For the more general Problem \ref{prob:FGMotzkin}, in this paper we examine more basic questions.

Given a lattice $\Lambda\subset \mb{Z}^r$, the following are three natural questions of this type. Let us emphasize immediately that, as discussed below, the answer to some of these questions can depend on $r$ and $\Lambda$.
\begin{question}\label{Q:rationality}
Is the number $\md_{\mb{Z}^r/\Lambda}(B_\Lambda)$ rational? 
\end{question}
By Lemma \ref{lem:estiminf}, in principle we can give arbitrarily accurate upper bounds for $\md_{\mb{Z}^r/\Lambda}(B_\Lambda)$ of the form $\phi_{B_\Lambda}(F)/|F|$ for finite tiling sets $F\subset \mb{Z}^r/\Lambda$. Is any of these bounds \emph{exact}?
\begin{question}\label{Q:min}
Is the infimum in \eqref{eq:mdinf} \textup{(}for $\Gamma=\mb{Z}^r/\Lambda$ and $E=B_\Lambda$\textup{)} a minimum? 
\end{question}

Clearly, a positive answer to Question \ref{Q:min} for a given $\Lambda\leq \mb{Z}^r$ also answers Question \ref{Q:rationality} positively for this $\Lambda$. Another closely related matter is the following.
\begin{question}\label{Q:periodicity}
Is the supremum in \eqref{eq:FGMD2} \textup{(}for $\Gamma=\mb{Z}^r/\Lambda$ and $E=B_\Lambda$\textup{)} attained by a periodic $B_\Lambda$-avoiding set?
\end{question}
Here, by a periodic subset of $\Gamma$ we mean a (finite) union of cosets of a finite-index subgroup of $\Gamma$. Again, an affirmative answer to Question \ref{Q:periodicity} would solve Question \ref{Q:rationality} positively for the given $\Lambda\leq \mb{Z}^r$. Moreover, it is not hard to see that it would also solve Question \ref{Q:min} positively.

These questions are related to more widely known problems. For example, by Proposition \ref{prop:equivZkrefined}, Question \ref{Q:rationality} includes as a special case the question of whether Motzkin densities in $\mb{Z}^k$ are rational, and the latter in turn includes the question of whether maximal packing densities in $\mb{Z}^k$ are rational, mentioned as an open problem in \cite{Schmidt&Tuller} and more recently in \cite{F&K&S}. We discuss these relationships in more detail in Subsection \ref{subsec:multidimcase}. Moreover, through these connections it also follows that Question \ref{Q:periodicity} can have a \emph{negative} answer for some lattices $\Lambda$; we detail this also in Subsection \ref{subsec:multidimcase}, using the recent result of Greenfeld and Tao \cite{Greenfeld&Tao} on aperiodic tilings of $\mb{Z}^k$. Thus, the answer to Question \ref{Q:periodicity} depends indeed on $\Lambda$, and this leads to the problem of determining under what conditions on $r$ and $\Lambda$ the answer to this question is positive. For questions \ref{Q:rationality} and \ref{Q:min}, no instances with negative answers are known so far.

When $\Lambda$ is the trivial lattice $\{0\}$, the answer to Question \ref{Q:periodicity} is trivially affirmative, and Problem \ref{prob:FGMotzkin} is easily solved, obtaining $\md_{\mb{Z}^r/\Lambda}(B_\Lambda)=1/2$ (the application of this fact to Motzkin's problem in compact abelian groups was noted in \cite[Theorem 2.4]{CCRS}). Another case of the periodicity question in which the answer is trivially affirmative is when $\Lambda$ has full rank $r$, in which case the group $\mb{Z}^r/\Lambda$ is finite. In Sections \ref{sec:rank-1} and \ref{sec:rank-r-1} we address some initial non-trivial cases of the above questions, specifically, the cases rank$(\Lambda)=1$ (Section \ref{sec:rank-1}) and rank$(\Lambda)=r-1$ (Section \ref{sec:rank-r-1}), in both of which we give an affirmative answer to Question \ref{Q:periodicity} (hence also to Questions \ref{Q:min} and \ref{Q:rationality}). A combination of these facts will yield a proof of the following result, which implies Theorem \ref{thm:3rational} (via Theorem \ref{thm:trans}).

\begin{proposition}\label{prop:maindiscrete}
For any lattice $\Lambda \subset \mb{Z}^r$ with $r\leq 3$, the supremum $\md_{\mb{Z}^r/\Lambda}(B_\Lambda)=\sup\{\overline{\delta}_{\mc{F}}(A):A\subset\mb{Z}^r/\Lambda, (A-A)\cap B_\Lambda=\emptyset\}$ is attained by a periodic $B_\Lambda$-avoiding set.
\end{proposition}
In Section \ref{sec:Fourier} we focus on the case $G=\mb{T}$ of Problem \ref{prob:CAGMotzkin} to present bounds for $\md_{\mb{T}}(D)$.

\section{On Motzkin densities in $\mb{Z}^k$}\label{sec:multidimcase}
In this section we prove the following result, which enables Motzkin's problem in $\mb{Z}^k$ to be subsumed as a special case of Problem \ref{prob:FGMotzkin}.
\begin{proposition}\label{prop:equivZkrefined}
Let $r,k$ be positive integers with $r\geq k$. If $\Lambda$ is a lattice in $\mb{Z}^r$ of rank $r-k$ having a basis that is a primitive set in $\mb{Z}^r$, then there is a set $D=\{t_1,\ldots,t_r\}\subset \mb{Z}^k$ that generates $\mb{Z}^k$ and such that $\md_{\mb{Z}^r/\Lambda}(B_\Lambda)=\md_{\mb{Z}^k}(D)$. Conversely, given any set $D=\{t_1,\ldots,t_r\}\subset \mb{Z}^k$ generating $\mb{Z}^k$, there is a lattice $\Lambda \subset \mb{Z}^r$ of rank $r-k$, having a basis that is a primitive set in $\mb{Z}^r$, such that $\md_{\mb{Z}^r/\Lambda}(B_\Lambda)=\md_{\mb{Z}^k}(D)$. 
\end{proposition}

\begin{remark}\label{rem:reducdim}
The general Motzkin problem in $\mb{Z}^k$, where the missing-difference set $D$ is any finite set of elements of $\mb{Z}^k$, is easily reduced without loss of generality to the case where $D$ generates $\mb{Z}^k$. Indeed, if $D$ generates only a proper subgroup $S$ of $\mb{Z}^k$, then any $D$-avoiding set within $S$ can be copied and translated into the cosets of $S$, yielding a $D$-avoiding subset of $\mb{Z}^k$ of density in $\mb{Z}^k$ equal to the original $D$-avoiding set's density in $S$. Hence the problem reduces to finding $\md_S(D)$. Then, letting $B$ be an integral basis of cardinality $k$ for the lattice $S$, expressing the elements of $S$ in terms of $B$ yields an isomorphism $\varphi:S\to \mb{Z}^k$, which implies that $\md_S(D)=\md_{\mb{Z}^k}(\varphi(D))$, where the set $\varphi(D)$ generates $\mb{Z}^k$ as required.
\end{remark}

To prove Proposition \ref{prop:equivZkrefined}, we first gather some tools from linear algebra and basic geometry of numbers. Recall that a linearly independent set $\{b_1,\ldots,b_k\}\subset\mb{Z}^r$ is \emph{primitive in $\mb{Z}^r$} if $\Span_{\mb{R}}(\{b_1,\ldots,b_k\})\cap \mb{Z}^r = \Span_{\mb{Z}} \{b_1,\ldots,b_k\}$. We shall often abbreviate by saying that such a set is \emph{primitive}, as we will not need to consider primitivity in lattices other than $\mb{Z}^r$. We recall some facts about this notion.
\begin{lemma}\label{lem:primset}
Let $b_1,\ldots,b_k\in\mb{Z}^r$ be linearly independent over $\mb{Q}$, and let $B$ denote the $k\times r$ integer matrix whose $i$-th row-vector is $b_i$. The following properties are equivalent:
\setlength{\leftmargini}{0.8cm}
\begin{enumerate}
\item The set $\{b_1,\ldots,b_k\}$ is primitive.
\item The set $\{b_1,\ldots,b_k\}$ is a subset of a basis of $\mb{Z}^r$.
\item Every non-zero entry in the Smith Normal Form of $B$ equals 1.
\item The columns of $B$ generate $\mb{Z}^k$: $B\mb{Z}^r=\mb{Z}^k$.
\end{enumerate}
\end{lemma}
\begin{proof}
$(i)\Leftrightarrow (ii)$: for the non-trivial (forward) implication, see \cite[p.\ 21, Theorem 5]{G&L}.

$(ii)\Leftrightarrow (iii)$: this follows from standard results (e.g.\  \cite[p.\ 15, Lemma 2]{Cassels}) using that property $(iii)$ is equivalent to the determinants of $k\times k$ submatrices of $B$ being coprime.

$(ii)\Rightarrow (iv)$: fix any $y\in \mb{Z}^k$, and note that by $(ii)$ we can add bottom row vectors to $B$ to obtain a matrix $B^*\in \mb{Z}^{r\times r}$ whose rows form a basis, and therefore $B^*$ is unimodular. Extending $y$ to a vector $y^*\in \mb{Z}^r$ by the addition of arbitrary integer coordinates, it follows that there exists $x^*\in \mb{Z}^r$ such that $B^*x^*=y^*$. The projection of $x^*$ to the first $k$ coordinates is then an element $x\in \mb{Z}^k$ such that $Bx=y$. Hence $(iv)$ holds.

$(iv)\Rightarrow (iii)$: let $B=USV$ with $S$ in Smith Normal Form, where $U\in \mb{Z}^{k\times k}$, $V\in \mb{Z}^{r\times r}$  are unimodular, and $S=\big(\textrm{diag}(\alpha_1,\ldots,\alpha_k)~ | ~ 0^{k\times (r-k)}\big)$ with $\alpha_1\cdots\alpha_k$ equal to the greatest common divisor of the determinants of all $k\times k$ submatrices of $B$ (also known as the $k$-th determinantal divisor of $B$). If $(iii)$ fails, i.e.\ if $\alpha_i>1$ for some $i$, then $B\mb{Z}^r=US\mb{Z}^r$, where the $i$-th coordinate of every element of $S\mb{Z}^r$ lies in the proper subgroup $\alpha_i\mb{Z}$, so $S\mb{Z}^r$ is a proper subgroup of $\mb{Z}^k$, whence so is $B\mb{Z}^r=US\mb{Z}^r$, and therefore $(iv)$ fails.
\end{proof}

Recall that a vector subspace of $\mb{R}^r$ is said to be \emph{rational} if it admits a basis in $\mb{Q}^r$, or equivalently in $\mb{Z}^r$. We shall use the following fact.

\begin{lemma}\label{lem:primbasis}
Every rational subspace of $\mb{R}^r$ has a basis in $\mb{Z}^r$ that is a  primitive set.
\end{lemma}

\begin{proof}
Let $B\in \mb{Z}^{k\times r}$ be a matrix having as rows an integral basis of the subspace, thus the subspace is $\mb{R}^r B$. Consider the Smith Normal Form $B=USV$. Then the rows of $U^{-1}B=SV$ also form an integral basis of this subspace (since $U$ is invertible). 
Changing all the entries $\alpha_i$ in the diagonal of $S$ to 1, and naming the resulting matrix $S'$, we have that $S'V$ is still an integral basis of this subspace (indeed it is the basis obtained by dividing the $i$-th row of $SV$ by $\alpha_i$, for each $i\in [k]$). The rows of $S'V$ form a primitive set in $\mb{Z}^r$ by Lemma 3.1 $(iii)$.
\end{proof}

We can now prove the main result.

\begin{proof}[Proof of Proposition \ref{prop:equivZkrefined}]
Let $d=r-k$. To see the first implication, let $L\in \mb{Z}^{d\times r}$ be a matrix whose rows form a primitive basis of $\Lambda$, so $\Lambda=\mb{Z}^d L$, and let $L=U S V$ be the Smith Normal Form of $L$, where  $S=(I_d|0^{d\times k})$ by primitivity and Lemma \ref{lem:primset}. Let $\varphi_V$ be the isomorphism $\mb{Z}^r/\Lambda \to \mb{Z}^r/\mb{Z}^d S$, $x+\Lambda \mapsto xV^{-1} + \Lambda V^{-1} =  xV^{-1} + \mb{Z}^dS$. Then $A\subset \mb{Z}^r/\Lambda$ is $B_\Lambda$-avoiding if and only if $\varphi_V(A)$ is $D$-avoiding for $D=\varphi_V(B_\Lambda)$. Given the diagonal form of $S$, it is also clear that $\mb{Z}^r/\mb{Z}^dS\cong\mb{Z}^k$. Moreover, since $B_\Lambda$ clearly generates $\mb{Z}^r/\Lambda$, the set $D$ generates $\mb{Z}^k$. Finally, as F\o lner sequences in $\mb{Z}^r/\Lambda$ are mapped to F\o lner sequences in $\mb{Z}^k$ and vice versa by $\varphi_V,\varphi_V^{-1}$, we conclude that  $\md_{\mb{Z}^r/\Lambda}(B_\Lambda)=\md_{\mb{Z}^k}(D)$ as claimed.

We now prove the converse. Let $B\in\mb{Z}^{k\times r}$ be the matrix with columns $t_1,\ldots,t_r$. By Lemma \ref{lem:primbasis} the subspace $\ker(B)$ of $\mb{R}^r$ has a  basis, consisting of $r-k=d$ integer vectors, which is a primitive set in $\mb{Z}^r$. Let $L\in\mb{Z}^{d\times r}$ have these basis vectors as rows, and let $\Lambda$ be the lattice $\mb{Z}^d L$. We shall now see that $\md_{\mb{Z}^r/\Lambda}(B_\Lambda)=\md_{\mb{Z}^k}(D)$, which will complete the proof. To see this, note first that since the rows of $L$ form a primitive set, their span $\ker(B)$ satisfies $\ker(B)\cap \mb{Z}^r=\Lambda$. Moreover, by assumption the columns of $B$ generate $\mb{Z}^k$, so the image of the homomorphism $B:\mb{Z}^r\to \mb{Z}^k$ is  $\mb{Z}^k$. It follows from the first isomorphism theorem that there is an isomorphism $\varphi_B:\mb{Z}^r/\Lambda\to\mb{Z}^k$ such that $B=\varphi_B\circ \pi_\Lambda$, where $\pi_\Lambda: \mb{Z}^r\to \mb{Z}^r/\Lambda$ is the canonical quotient map. Note that $\varphi_B$ maps $B_\Lambda$ to $D$. This implies, similarly as in the previous paragraph, that $\md_{\mb{Z}^r/\Lambda}(B_\Lambda)=\md_{\mb{Z}^k}(D)$.
\end{proof}

\subsection{Connections with tiling and packing problems}\label{subsec:multidimcase}\hfill\smallskip\\
In this subsection we briefly discuss the relation between the property of $A\subset \mb{Z}^k$ being a $D$-avoiding set, and the stronger property of $A$ being a \emph{tiling complement} in $\mb{Z}^k$ of some finite set $F$. Recall that given an abelian group $\Gamma$, we denote by $\mc{T}_\Gamma$ the set of finite subsets that tile $\Gamma$. 
\begin{lemma}\label{lem:optimatile}
Let $F\in \mc{T}_{\mb{Z}^k}$, and let $D=(F-F)\setminus \{0\}$. Then $\md_{\mb{Z}^k}(D)=\frac{1}{|F|}$.
\end{lemma}
\begin{proof}
By \eqref{def:FGMD1} it suffices to find a F\o lner sequence $(F_N)_{N\in\mb{N}}$ in $\mb{Z}^k$ such that
\[
\lim_{N\to\infty}\phi_D(F_N)/|F_N|=1/|F|.
\]
Letting $B$ be a tiling complement of $F$, we can form such a sequence with the convenient property that each set $F_N$ is also tiled by $F$, for example
\[
F_N=\bigsqcup_{x\in B:\, x+F\,\subset\, [-N,N]^k} (x+F).
\]
Now, with this sequence $(F_N)_N$, for every $N$ such that $F_N\neq \emptyset$, any $D$-avoiding set $A\subset F_N$ must satisfy $|A|/ |F_N| \leq 1/|F|$, since otherwise $A$ would have density greater than $1/|F|$ inside some tile $x+F$ ($x\in B$), yielding distinct $t,t'\in A\cap (x+F)$, so $A-A$ would contain $t-t'\in D$, contradicting that $A$ is $D$-avoiding. Hence $\phi_D(F_N)/|F_N|\leq 1/|F|$. On the other hand, the set $B\cap F_N$ is clearly $D$-avoiding (if $|B\cap F_N|=0$ or 1 this is trivial, and otherwise there would be distinct $x,x'\in B$ and $t,t'\in F$ such that $x-x'=t-t'$, implying that $(x+F)\cap (x'+F)\neq \emptyset$, contradicting the tiling property), whence $\phi_D(F_N) /|F_N|\geq 1/|F|$ whenever $F_N\neq \emptyset$. Thus $\phi_D(F_N)/|F_N|=1/|F|$ for large $N$, and our claim follows.
\end{proof}
Given a subset $F$ of an abelian group, a set $A$ in this group is said to be an \emph{$F$-packing} (or a \emph{packing complement} for $F$) if $a+F$ and $a'+F$ are disjoint for any $a\neq a'$ in $A$.
\begin{defn}
Let $F\in\mc{T}_{\mb{Z}^k}$. We say that a set $A\subset \mb{Z}^k$ \emph{tightly packs $\mb{Z}^k$ with $F$} (or that $A$ is a \emph{tight packing complement} of $F$ in $\mb{Z}^k$) if $A$ is an $F$-packing (equivalently $A$ is $D$-avoiding for $D=(F-F)\setminus\{0\}$) and $\limsup_{N\to\infty} \frac{|A\cap [-N,N]^k|}{(2N+1)^k}=\frac{1}{|F|}$.
\end{defn}

Clearly, if $A$ is a tiling complement for $F$ in $\mb{Z}^k$ then $A$ is also a tight packing complement of $F$ in $\mb{Z}^k$. The converse is false, indeed if $A'$ is a tiling complement of $F$, then by deleting any finite number of points from $A'$ we obtain a set $A$ that is no longer a tiling complement, yet $A$ still tightly packs $\mb{Z}^k$ with $F$. However, we have the following fact, indicated to us by Mihalis Kolountzakis.

\begin{lemma}\label{lem:tpc-tc}
If $A$ is a \emph{periodic} tight-packing complement of $F\in \mc{T}_{\mb{Z}^k}$, then $A$ is a tiling complement of $F$.
\end{lemma}
\begin{proof}
Being periodic, the set $A$ has a density, i.e.\ the limit 
$\delta(A):=\lim_{N\to\infty}  \frac{|A_N|}{(2N+1)^k}$ exists, where $A_N:=A\cap [-N,N]^k$. Since $A$ is a tight packing complement of $F$, we have $\delta(A)=1/|F|$.  The set $A+F$ is also periodic, so it also has a well-defined density $\delta(A+F)$. 

It follows from the tight-packing assumption that $\delta(A+F)=1$. Indeed, for any $\epsilon>0$, for all large $N$ (depending on $\epsilon$, $F$) we have $|(A+F)_N\Delta (A_N+F)|\leq \epsilon (2N+1)^k$, so $|(A+F)_N|/(2N+1)^k\geq |A_N+F|/(2N+1)^k-\epsilon\geq |A_N||F|/(2N+1)^k-\epsilon$, whence, letting $N\to\infty$, we deduce that $\delta(A+F)\geq |F|\,\delta(A)-\epsilon=1-\epsilon$.  Letting $\epsilon\to 0$, our claim follows. 

Suppose for a contradiction that $A$ is not a tiling complement of $F$, so there exists $x\in \mb{Z}^k\setminus (A+F)$. Then, letting $S$ be a full-rank lattice of periods of $A+F$, the entire coset $x+S$ is disjoint from $A+F$. Moreover, the periodic set $x+S$ has density $\delta(x+S)=1/d$ where $d$ is the index of the subgroup $S$. As density is an additive set function, we have $1\geq \delta((A+F)\cup (x+S))= \delta(A+F) + \delta(x+S)=1+1/d$, a contradiction.
\end{proof}
The well-known periodic tiling conjecture asserts that a set in $\mc{T}_{\mb{Z}^d}$ must necessarily have some \emph{periodic} tiling complement (see \cite{Greenfeld&Tao} and its bibliography for background, in particular \cite{La&Wa}). In a recent breakthrough, Greenfeld and Tao obtained a counterexample to this conjecture \cite{Greenfeld&Tao}. Combining this with the basic facts above, we deduce the following result, showing that the answer to Question \ref{Q:periodicity} can be negative.
\begin{proposition}\label{prop:NegPeriodQ}
For some $k\in \mb{N}$, there exists a finite set $D\subset \mb{Z}^k\setminus\{0\}$ such that the Motzkin density $\md_{\mb{Z}^k}(D)$ is attained only by  $D$-avoiding sets that are not periodic.
\end{proposition}
\begin{proof}
First note that, as established in \cite[Corollary 1.6]{Greenfeld&Tao}, for sufficiently large $k$ there exists a tile $F\in \mc{T}_{\mb{Z}^k}$ such that no tiling complement of $F$ is periodic. Combining this with Lemma \ref{lem:tpc-tc}, we deduce that every tight packing complement of $F$ is not periodic. By Lemma \ref{lem:optimatile}, the set $D=(F-F)\setminus\{0\}$ satifies $\md_{\mb{Z}^k}(D)=1/|F|$, but this optimal density cannot be attained by any periodic $D$-avoiding set, since such a set would be a tight packing complement of $F$.
\end{proof}

\begin{remark}
Other results on periodic tilings can similarly imply answers to Question \ref{Q:periodicity} in various special cases. For instance, in the positive direction, let us consider cases in which the periodic tiling conjecture \emph{does} hold. If this conjecture holds\footnote{It is not yet known precisely for which values of $d$ this conjecture holds in $\mb{Z}^d$. The conjecture is known to hold for $d=1,2$ and known to fail for sufficiently large $d$ (see \cite{Greenfeld&Tao}).} in $\mb{Z}^d$, then for any set $D=(F-F)\setminus\{0\}$ for some $F\in\mc{T}_{\mb{Z}^d}$, the corresponding Question \ref{Q:periodicity} (via Proposition \ref{prop:equivZkrefined}) has a positive answer (and, by Lemma \ref{lem:optimatile}, we also have $\md_{\mb{Z}^d}(D)=1/|F|$).

Similarly, results in the topics of optimal packing and covering densities can have direct relevance in the topic of Motzkin densities, and some basic notions are naturally shared by these topics; see for instance \cite{Schmidt&Tuller2} (in particular a notion in the discrete setting that is similar to the lattice $\Lambda$ from \eqref{eq:Lambda}, called the \emph{relation lattice}); see also \cite{BoJaRi}. 
\end{remark}

\begin{remark}
Conversely, positive answers to Question \ref{Q:periodicity} in specific cases yield confirmations of corresponding cases of the periodic tiling conjecture. To detail this, given any finite tile $F\subset \mb{Z}^k$, let $D$ be formed by taking precisely one element of $\{t,-t\}$ for every $t\in (F-F)\setminus\{0\}$, and let $r=|D|$. From the proof of Proposition \ref{prop:equivZkrefined}, it follows that, for the lattice $\Lambda\subset \mb{Z}^r$ given by that proposition, a periodic subset of $\mb{Z}^r/\Lambda$ attaining $\md_{\mb{Z}^r/\Lambda}(B_\Lambda)$ yields (via the isomorphism $\varphi_B$), a periodic tiling complement for $F$. Hence a positive answer to Question \ref{Q:periodicity} for this $\Lambda$ and $r$ confirms the periodic tiling conjecture for $F$. In particular, if $D\subset \mb{Z}^k$ consists of $k$ or $k+1$ elements that generate $\mb{Z}^k$, then this correspondence combined with Theorem \ref{thm:rank1} confirms that $F$ has a periodic tiling complement. For $|D|\geq k+2$, our results in this paper do not help in this way, and this motivates further work on Question \ref{Q:periodicity}. For instance, by the above correspondence, if Question \ref{Q:periodicity} could be answered positively for every $r$ and $\Lambda\subset \mb{Z}^r$ of rank $r-k$ having a primitive basis, then the periodic tiling conjecture in $\mb{Z}^k$ would be confirmed. %Note that for the tiling conjecture for $F$, can we assume that $F-F$ spans $\mb{Z}^k$. Indeed, first note that, since tiling is translation invariant, we may assume that $F$ contains 0, so that $F$ is a subset of $F-F$. Now, if $F-F$ only spans a *proper* subgroup $S$, then since $F\subset F-F\subset S$, given a tiling complement $A$ of $F$ in Z^k, we have that $A':=A\cap S$ satisfies $A'+F= A\cap S + F\cap S=(A+F)\cap S=S$, so $A'$ is a tiling complement of $F$ inside $S$, and so if we can find a periodic tiling complement inside $S$, we can then extend it to a periodic tiling complement in $\mb{Z}^k$. Therefore we can suppose that $S=\mb{Z}^k$, i.e. that $F-F$ generates $\mb{Z}^k$.
\end{remark}

\section{Motzkin densities in the case rank$(\Lambda)=1$}\label{sec:rank-1}

\noindent In this section we present the following solution to Problem \ref{prob:FGMotzkin} for lattices $\Lambda$ of rank 1, which implies Theorem \ref{thm:rk1-intro} (via Theorem \ref{thm:trans}).

\begin{theorem}\label{thm:rank1}
Let $\Lambda$ be a cyclic subgroup of $\mb{Z}^r$ generated by an element $m=(m_1,\ldots,m_r)\in \mb{Z}^r\setminus\{0\}$, and let $k=\|m\|_{\ell^1}$. Then
\begin{equation}\label{eq:rank1formula}
\md_{\mb{Z}^r/\Lambda}(B_\Lambda)= \frac{\lfloor k/ 2\rfloor}{k}.
\end{equation}
\end{theorem}

The case of this theorem for $r=2$ was obtained in \cite[Theorem 3.2]{CCRS}, and includes as a special case (via Proposition \ref{prop:equivZkrefined}) the classical Cantor--Gordon result \cite[Theorem 4]{C&G}.

Given the lattice $\Lambda\subset \mb{Z}^r$, of rank $d$, let us say that a lattice $S\subset \mb{Z}^r$ is \emph{complementary to} $\Lambda$ if the linear span $\Span_{\mb{Q}}(S)$ is a complement of $\Span_{\mb{Q}}(\Lambda)$ (i.e.\ $\Span_{\mb{Q}}(S)\oplus \Span_{\mb{Q}}(\Lambda)=\mb{Q}^r$), equivalently rank$(S)=r-d$ and $S\cap \Lambda=\{0\}$. For any such lattice $S$, let $\mc{P}_S$ denote\footnote{We include only $S$ as a subscript in ``$\mc{P}_S$", not $\Lambda$, because we think of $\Lambda$ as fixed from the start, whereas we may change the set $S$ to find the most convenient one for our purposes.} the fundamental parallelepiped of the rank $r$ lattice $\Lambda\oplus S$, and let us define
\begin{equation}\label{eq:FDC}
C_S:= \mc{P}_S\cap \mb{Z}^r.
\end{equation}
It is well known that the cardinality $|C_S|$ is the number of cosets of $S\oplus \Lambda$ in $\mb{Z}^r$. Moreover $C_S$ tiles $\mb{Z}^r$:
\begin{equation}\label{eq:gen-tile}
\mb{Z}^r = \bigsqcup_{w\in S\oplus\Lambda} (w+C_S).
\end{equation}
\begin{lemma}\label{lem:Wtile}
The following set is a fundamental domain for the action of $\Lambda$ on $\mb{Z}^r$:
\begin{equation}\label{eq:Wtile}
W_S:= \bigsqcup_{w\in S} (w+C_S).
\end{equation}
In particular, the set $C_S$ tiles $\mb{Z}^r/\Lambda$ \textup{(}viewing $\mb{Z}^r/\Lambda$ as $W_S$ with addition modulo $\Lambda$\textup{)}.
\end{lemma}
\begin{proof}
Given any $x\in \mb{Z}^r$, since $C_S$ is a fundamental domain for the action of $S\oplus \Lambda$ on $\mb{Z}^r$, there is a unique triple $(s,\lambda,c)\in S\times \Lambda\times C_S$ such that $x=s+\lambda+c$, implying that $x$ has the representative $s+c\in W_S$ modulo $\Lambda$. Moreover, this representative is unique, otherwise there would be $s'\in S$ and $c'\in C_S$ such that $s'+c'\neq s+c$ and, for some $\lambda'\in \Lambda$, we have $x=s'+\lambda'+c'$. Thus 
$(s,\lambda,c)$ and $(s',\lambda',c')$ would be two different triples in $S\times \Lambda\times C_S$ with sum equal to $x$, a contradiction.
\end{proof}
Given a finite graph $H$, recall that the \emph{independence number} of $H$, denoted by $\alpha(H)$, is the maximum of the cardinalities of sets of vertices of $H$ that are \emph{independent} (i.e.\ such that no pair of vertices in these sets form an edge of $H$). Another way to phrase Problem \ref{prob:FGMotzkin} is as the estimation of the supremum of upper densities of independent sets in the Cayley graph\footnote{Two vertices $x+\Lambda,y+\Lambda$ in $\mb{Z}^r/\Lambda$ are adjacent in this graph $G$ if and only if there is $j\in [r]$ and $\lambda\in \Lambda$ such that one of the two elements $x-y,y-x\in \mb{Z}^r$ equals $e_j+\lambda$.}
\begin{equation}\label{eq:CG}
G:=\textrm{Cay}(\mb{Z}^r/\Lambda, B_\Lambda).
\end{equation}
For the subadditive function $\phi_{B_\Lambda}$ (recall  \eqref{eq:subadfn}), the quantity $\phi_{B_\Lambda}(C_S)$ is the independence number of the subgraph $G[C_S]$ of $G$ induced by the vertex subset $C_S$. We define the following two subsets $V_0,V_1\subset \mb{Z}^r$:
\begin{equation}\label{eq:2classes}
V_i=\{x\in \mb{Z}^r: x_1+\cdots+x_r \equiv i\!\!\mod 2\},\; i\in \{0,1\}.
\end{equation}
Note that these sets are independent in the graph $G'=\textrm{Cay}(\mb{Z}^r,\{e_1,\ldots,e_r\})$, and therefore so are the sets $V_0\cap W_S$, $V_1\cap W_S$. 

To prove Theorem \ref{thm:rank1}, we answer Question \ref{Q:periodicity} positively in this rank 1 case with an explicit construction. For technical convenience we use the following reductions.
\begin{remark}\label{rem:reducs}
It suffices to prove the special case of Theorem \ref{thm:rank1} where each $m_i$ is non-zero. Indeed, in the general case, we may assume without loss of generality that $m_1,\ldots,m_j$ are non-zero and the rest are 0. Then, by the above special case applied in $\mb{Z}^j$, we would obtain an $E_j$-avoiding set $A_j\subset \mb{Z}^j/(\Lambda\cap \mb{Z}^j)$ of density $\frac{\lfloor k/ 2\rfloor}{k}$, where $E_j=\{\overline{e_1},\ldots,\overline{e_j}\}$. Taking appropriately-shifted copies of $A_j$ along the $(j+1)$-st dimension, we would then obtain an $E_{j+1}$-avoiding subset of $\mb{Z}^{j+1}/(\Lambda\cap \mb{Z}^{j+1})$ still of density $\frac{\lfloor k/ 2\rfloor}{k}$. Continuing this way adding dimensions, we would end up with a $B_\Lambda$-avoiding subset of $\mb{Z}^r/\Lambda$ as required. 

Note also that it suffices to prove \eqref{eq:rank1formula} under the additional assumption that $k$ is odd. Indeed, if $k$ is even then the graph $G$ in \eqref{eq:CG} can be partitioned into the two classes of vertices $V_i\cap W_S$, $i=0,1$ (for $V_i$ in \eqref{eq:2classes}). Any of these two classes yields an $E$-avoiding set of density $1/2$, which proves \eqref{eq:rank1formula} in the even case.

Finally, we can assume that each $m_i$ is positive, because the sign change amounts to replacing $\Lambda$ by an isomorphic lattice, and replacing the missing difference $e_i+\Lambda$ by $-e_i+\Lambda$, which does not alter the Motzkin density.
\end{remark}
Thus we start with a generator $m=(m_1,\ldots,m_r)\in \mb{Z}_{>0}^r$ of $\Lambda$, and choose the following lattice $S$ complementary to $\Lambda$:
\begin{eqnarray}\label{eq:rank1S}
&\textrm{For each $i\in [r-1]$ we define $s^{(i)}:=e_{i+1}-e_i$,}&\nonumber \\ 
&\textrm{and we define $S:=\mb{Z}s^{(1)}+\cdots+\mb{Z}s^{(r-1)}$.}&
\end{eqnarray} 

\begin{lemma}\label{lem:tile}
The lattice $S$ in \eqref{eq:rank1S} is complementary to $\Lambda$ and the fundamental domain $C_S$ of $\Lambda\oplus S$ has cardinality $k=\|m\|_{\ell^1}$.
\end{lemma}

\begin{proof}
By partial summation, any $n\in \mb{Z}^r$ satisfies $n=(u\cdot n)e_r-v$ with $u=(1,\ldots,1)$ and $v=\sum_{i=1}^{r-1} (n_1+\cdots + n_i)s^{(i)}\in S$. In particular we have $m=w+\|m\|_{\ell^1}e_r$ with $w\in S$. Let 
 $q\in \mb{Z}$ and $a\in [0,\|m\|_{\ell^1}-1]$ be respectively the quotient and the remainder when we divide $u\cdot n$ by $\|m\|_{\ell^1}$. It follows that $n=q m - (v+qw)+a e_r\in\Lambda+S+a e_r$. As $w$ is fixed and $q$ and $a$ are uniquely determined, we get at once that $\Lambda,S$ are complementary and also the claim about $C_S$, in fact we deduce that $\mb{Z}^r/(\Lambda\oplus S)\cong  \mb{Z}_{\|m\|_{\ell^1}}$.
\end{proof}
To prove Theorem \ref{thm:rank1}, we shall split the argument into the following two lemmas.
\begin{lemma}\label{lem:rank1ub}
We have $\md_{\mb{Z}^r/\Lambda}(B_\Lambda)\leq \frac{\lfloor k/2\rfloor}{k}$.
\end{lemma}
\begin{proof}
By Remark \ref{rem:reducs}, we can assume that $k$ is odd. By Lemma \ref{lem:tile} and \eqref{eq:mdinf} we have $\md_{\mb{Z}^r/\Lambda}(B_\Lambda)\leq \frac{\alpha(G[C_S])}{|C_S|}=\frac{\alpha(G[C_S])}{k}$. We shall prove that $\alpha(G[C_S])= \lfloor k/2\rfloor$ by showing that the graph $G[C_S]$ is isomorphic to a cycle of odd length $k$ (the result will then follow by the standard fact that the independence number of an odd cycle of length $k$ is $\lfloor k/2\rfloor$). 

Let $u=(1,\ldots,1)\in\mb{Z}^r$, and let $f:\mb{Z}^r\to\mb{Z}$,  $x\mapsto u\cdot x$. Let
\[
\wt{C}_S:=\Big\{x\in\mb{Z}^r:x=\sum_{j=1}^{r-1}\lambda_j s^{(j)}+\lambda_r m,\;\lambda_1,\ldots,\lambda_{r-1}\in [0,1), \lambda_r\in \mb{R}\Big\}.
\]
For every $x\in \wt{C}_S$, since $u$ is orthogonal to every $s^{(i)}$, we have $f(x)=\lambda_r u\cdot m=\lambda_r k$, and since this is an integer we must have $\lambda_r=j/k$ for some $j\in\mb{Z}$. Note that $f$ is bijective $\wt{C}_S\to \mb{Z}$, since given any $j\in\mb{Z}$, its unique preimage $x\in \wt{C}_S$ under $f$ is the solution to
\[
(\lambda_1+\tfrac{j}{k}m_1,\; \lambda_2-\lambda_1+\tfrac{j}{k}m_2,\ldots,\; \lambda_{r-1}-\lambda_{r-2}+\tfrac{j}{k}m_{r-1},\;\tfrac{j}{k}m_r-\lambda_{r-1})\in \mb{Z}^r.
\]
Indeed, there are unique values of $\lambda_1\ldots,\lambda_{r-1}\in [0,1)$ such that the first $r-1$ entries are integers $n_1,\ldots,n_{r-1}$, and once these values have been fixed, the last entry turns out to be automatically in $\mb{Z}$, since we have
\begin{eqnarray*}
\tfrac{j}{k}m_r-\lambda_{r-1} & = & \tfrac{j}{k}m_r +\tfrac{j}{k}m_{r-1} -n_{r-1}-\lambda_{r-2} \\
& = & \tfrac{j}{k}m_r +\tfrac{j}{k}m_{r-1} + \tfrac{j}{k}m_{r-2} -n_{r-1}-n_{r-2}-\lambda_{r-3}\\
& = & \cdots \; = \; \tfrac{j}{k} u\cdot m - n_{r-1}-n_{r-2}-\cdots -n_1 \in \mb{Z}.
\end{eqnarray*}
Having proved that $f$ is a bijection, let us now note that if $x,y\in \wt{C}_S$ satisfy $f(y)=f(x)+1$ then $y=x+e_i$ for some $i\in [r]$. Indeed $x=\sum_{j=1}^{r-1}\lambda_j s^{(j)} + \lambda_r m$ and $y=\sum_{j=1}^{r-1}\lambda_j' s^{(j)}+ \lambda'_r m$, where $\lambda_r=j/k$ and $\lambda_r'=(j+1)/k$, so $y-x$ is an element of $\mb{Z}^r$ of the form $\beta_1 s^{(1)}+\cdots+\beta_{r-1} s^{(r-1)}+\frac{1}{k} m$ with $\beta_i=\lambda_i'-\lambda_i\in (-1,1)$. Since $\frac{1}{k} m$ is an element in the unit ball of the $\ell^1$-norm with all coordinates positive, it follows that $y-x$ must be an element $e_i$ of the standard basis. 

We have thus proved that $\textrm{Cay}(\wt{C}_S,\{e_i\}_{i\in [r]})$ is isomorphic to $\textrm{Cay}(\mb{Z},\{\pm 1\})$, and it is then clear that $G[C_S]$ is isomorphic to the $k$-cycle.
\end{proof}

\begin{lemma}\label{lem:rank1lb}
We have $\md_{\mb{Z}^r/\Lambda}(B_\Lambda)\geq \frac{\lfloor k/2\rfloor}{k}$.
\end{lemma}
To prove this we use the fundamental domain $W_S$ from \eqref{eq:Wtile}, noting that
\[
W_S=\{v\in \mb{Z}^r:u\cdot v\in [0,k-1]\}.
\]
Indeed this can be seen by noting that $W_S=H\cap \mathbb{Z}^r$ where $H$ is the slice of $\mb{R}^r$ between the hyperplanes $\Span_{\mb{R}}(S)=u^\perp$ and $u^\perp+m$, i.e.\ $H= \{v\in \mb{R}^r: 0\leq u\cdot v< k=\|m\|_{\ell^1}\}$.
\begin{proof}[Proof of Lemma \ref{lem:rank1lb}]
The idea is to form a $B_\Lambda$-avoiding set $A$ of density $\frac{\lfloor k/2\rfloor}{k}$ consisting of points in $W_S$ whose sum of coordinates is even. However, there are technicalities, as not all such points in $W_S$ can be included. More precisely, let
\[
A:= \big\{v\in W_S: u\cdot v\equiv 0\!\!\mod 2\textrm{ and } u\cdot v\neq k-1\big\}.
\] 
Let us first prove that $A$ is a $B_\Lambda$-avoiding subset of $\mb{Z}^r/\Lambda$. For this purpose let
\begin{equation}\label{eq:P-rank1}
P:=\{c\in C_S: u\cdot c\equiv 0\!\!\mod 2\textrm{ and }c\cdot u\neq k-1\}.
\end{equation}
Note that (using \eqref{eq:Wtile}) we have $A=\big(\bigsqcup_{s\in S} P+s\big)\!\mod \Lambda$, and that
\begin{equation}\label{eq:even}
\textrm{the difference set $P-P$ in $\mb{Z}^r$, and the set $S$, are both subsets of $V_0$.}
\end{equation} 
Suppose now for a contradiction that two points $x,y\in A$ differ by some element of $B_\Lambda$, say $y=x+\overline{e_j}$, where $j\in [r]$. Thus $x$ and $y$ lie respectively in translates $s_x+P$ and $s_y+P$ with $s_x,s_y\in S$. By our assumption we then have $y=x+ e_j + tm$ (with addition in $\mb{Z}^r$) for some $t\in \mb{Z}$. We must then have $t=0$. Indeed, on one hand we have $|u\cdot (y-x)|= |u\cdot e_j + t\|m\|_{\ell^1}| = |1+ tk|\geq |t|\, k-1$, and on the other hand, since $x=s_x+p_x$ and $y=s_y+p_y$ for some $p_x,p_y\in P$, we have $|u\cdot (y-x)|= |u\cdot (p_y-p_x)|$, which is an even integer at most $2(\lfloor k/2\rfloor -1)=k-3$ (using the assumption that $k$ is odd). We thus deduce that $|t|\, k-1\leq k-3$, which is possible only if $t=0$ as claimed. We therefore have $y=x+ e_j$, with addition in $\mb{Z}^r$. But this implies that the element $s_y+p_y-s_x-p_x\in \mb{Z}^r$ equals $e_j$, which is impossible because this is an element of the set $S+P-P$, included in $V_0$ (by \eqref{eq:even}), so this element cannot equal $e_j\in V_1$. Hence $A$ is indeed $B_\Lambda$-avoiding. 

It only remains to prove that $A$ has density $\frac{\lfloor k/2\rfloor}{k}$. This is done  straightforwardly using that $A$ is a periodic set, tiled by the finite set $P$ (with period set $S$), so this density is equal to the density of $P$ inside $C_S$, namely $\frac{\lfloor k/2\rfloor}{k}$.
\end{proof}
\noindent Combining Lemma \ref{lem:rank1lb} and Lemma \ref{lem:rank1ub} we immediately deduce Theorem \ref{thm:rank1}.

\medskip

The sets $V_0,V_1$ from \eqref{eq:2classes} are independent sets in $G'=\textrm{Cay}(\mb{Z}^r,\{e_1,\ldots,e_r\})$. The sets $V_0\cap W_S,V_1\cap W_S$ may not be independent in the richer graph $G$ from \eqref{eq:CG}, in which edges modulo $\Lambda$ also appear. However, it seems natural to consider whether we can obtain an optimal $B_\Lambda$-avoiding set by eliminating adequately just one vertex from every such problematic edge modulo $\Lambda$. 
\begin{question}[Greedy construction]\label{Q:explicit}
Can a $B_\Lambda$-avoiding subset of $\mb{Z}^r/\Lambda$ of optimal upper density $\md_{\mb{Z}^r/\Lambda}(B_\Lambda)$ always be obtained by taking one of the two sets $V_0\cap W_S$, $V_1\cap W_S$, and removing from this set one vertex from each edge in $G$? 
\end{question}

\section{The case rank$(\Lambda)=r-1$}\label{sec:rank-r-1}

\noindent For any finite set $D\subset \mb{Z}$, Cantor and Gordon proved in \cite[Theorem 5]{C&G} that the Motzkin density $\md_{\mb{Z}}(D)$ is always attained by a \emph{periodic} subset of $\mb{Z}$ whose period is bounded solely in terms of $D$. In this section we extend this result to the case rank$(\Lambda)=r-1$ of Problem \ref{prob:FGMotzkin}. This is indeed an extension in the sense that, by Proposition \ref{prop:equivZkrefined}, the original Motzkin problem in $\mb{Z}$ is strictly subsumed in the case rank$(\Lambda)=r-1$ of Problem \ref{prob:FGMotzkin}.

In this section it will be convenient to use the Fundamental Theorem of finitely generated abelian groups (see Remark \ref{rem:NatFolner}), which yields an isomorphism $\varphi$ from our initial group $\mb{Z}^r/\Lambda$ (with rank$(\Lambda)=r-1$) to $G\times \mb{Z}$ for some finite abelian group $G$, so that the original Motzkin density $\md_{\mb{Z}^r/\Lambda}(B_\Lambda)$ in this case equals $\md_{G\times \mb{Z}}(E)$ for $E=\varphi (B_\Lambda)$.

A set $A\subset G\times \mb{Z}$ is \emph{periodic} if there exists a positive integer $t$ such that $A+(0_G,t)=A$; we then call $t$ a \emph{period} of $A$. The main result in this section is then the following. 
\begin{theorem}\label{thm:corank1}
Let $G$ be a finite abelian group, let $\Gamma=G\times \mb{Z}$, and let $E$ be a finite subset of $\Gamma$. Then there exists a periodic $E$-avoiding subset of $\Gamma$ of optimal density $\md_\Gamma(E)$.
\end{theorem}
Using $\varphi$ above, Theorem \ref{thm:corank1} answers Question \ref{Q:periodicity} positively for rank$(\Lambda)=r-1$.
\begin{proof}
This is a simple generalization of the proof of \cite[Theorem 5]{C&G}. 

Changing the sign of elements of $E$ if necessary, we can assume that all of them have non-negative $\mb{Z}$-component. Let $R$ be even and sufficiently large so that $E\subset G\times [0,R-1]$. 

Let $\mc{F}$ denote the tiling F\o lner sequence in $\Gamma$ with $N$-th term the set $F_N:=G\times [-RN,RN-1]$, $N\in \mb{N}$.

First we prove that $\md_{\Gamma}(E)$ is approximated arbitrarily closely by densities of \emph{periodic} $E$-avoiding sets $P\subset \Gamma$, i.e.\
\begin{equation}\label{eq:Mdperiodic}
\md_{\Gamma}(E)=\sup \{\delta_\mc{F}(P): P\subset \Gamma \textrm{ is $E$-avoiding and periodic}\},
\end{equation} 
where $\delta_\mc{F}(P)$ is the density $\lim_{N\to\infty} |F_N\cap P|/|F_N|$ (this limit exists when $P$ is periodic).

Fix any $\varepsilon>0$. By \eqref{eq:FGMD2}, there is an $E$-avoiding set $A\subset \Gamma$ with $\overline{\delta_\mc{F}}(A)> \md_{\Gamma}(E) - \varepsilon$. By definition of $\overline{\delta_\mc{F}}(A)$, for infinitely many $N$ we have $|A\cap F_N|/|F_N|> \md_{\Gamma}(E) -\varepsilon$. Fix any such $N$ satisfying also $|F_N|/(|F_N|+R|G|) \geq 1-\varepsilon$. We then define a periodic $E$-avoiding set $P\subset \Gamma$ as follows:
\begin{itemize}
\item $P\cap F_N=A\cap F_N$,
\item $P\cap \big(\cup_{j\in [-R(N+\frac{1}{2}),R(N+\frac{1}{2})-1]\setminus [-RN,RN-1])} G\times \{j\}\big)=\emptyset$,
\item $P$ is periodic with period $R(2N+1)$.
\end{itemize}
Then $\delta_\mc{F}(P)=\frac{|A\cap F_N|}{|F_N|+R|G|}\geq \frac{|A\cap F_N|}{|F_N|}-\varepsilon> \md_{\Gamma}(E)-2\varepsilon$, and \eqref{eq:Mdperiodic} follows.

Now we claim that if $P\subset\Gamma$ is a periodic $E$-avoiding set of period $t>2^{R|G|}$, then there is another periodic $E$-avoiding set $P'\subset\Gamma$, of period strictly less than $t$, such that $\delta_\mc{F}(P')\geq \delta_\mc{F}(P)$. 

To prove this, fix any periodic $E$-avoiding set $P\subset\Gamma$ of period $t>2^{R|G|}$, and consider the sets $P_j:=P\cap (G\times [j,j+R-1])$ for $j\in [t]$. Translating any $P_j$ by $-(0_G,j)$ we obtain an $E$-avoiding subset of the finite set $G\times [0,R-1]$. Since there are at most $2^{R|G|}$ such subsets, the assumption $t>2^{R|G|}$ implies that there exist $i< j$ in $[t]$ such that $P_i$ and $P_j$ are translates of the same set. 
Let $A$ be the periodic set with fundamental cell $\bigsqcup_{k\in [i,j-1]} P_k$, and let $B$ be the periodic set that is the ``complement" of $A$ in $P$, i.e.\ $B$ has fundamental cell $\bigsqcup_{k\in [j,t+i-1]} P_k$. Now note that $A$ and $B$ are both $E$-avoiding. Indeed, if there was a difference $x-y$ in $(A-A)\cap E$ or in $(B-B)\cap E$, then $x,y$ could not lie in the same piece $P_k$ (as these are $E$-avoiding), and so, since $E\subset G\times [0,R-1]$, there would exist $k$ such that $x$ is in the piece $P_k$ and $y$ is in the next piece $P_{k+1\!\mod \ell}$ (where $\ell=j-i$ if we are in $A$ and $\ell=t-j+i$ if we are in $B$). By construction, any two consecutive pieces in $A$ or $B$ correspond to identical consecutive copies in $P$, so we obtain the contradiction that $P$ is not $E$-avoiding. Having thus proved that $A$ and $B$ are both $E$-avoiding, now note that, by the pigeonhole principle, one of these sets has density at least $\delta_\mc{F}(P)$, so letting $P'$ be this set, we have proved our claim above. 

If this new periodic set $P'$ has period still greater than $2^{R|G|}$, we repeat the above argument. After a finite number of repetitions, we must eventually obtain an $E$-avoiding periodic set of period at most $2^{R|G|}$ and of density at least the density of the original set $P$. Combining this with \eqref{eq:Mdperiodic}, we deduce the following equality, which implies the theorem's conclusion:
\[
\md_{\Gamma}(E) = \max \big\{\delta_\mc{F}(P)\colon P\subset \Gamma \textrm{ is $E$-avoiding and periodic with period $t\leq 2^{R|G|}$}\big\}. \qedhere
\]
\end{proof}

We can now deduce the main result announced in the introduction.
\begin{proof}[Proof of Theorem \ref{thm:3rational}]
When the missing-difference set $D\subset \ab$ has cardinality $r\leq 3$, the case rank$(\Lambda)=r-1$ solved in this section, together with the case rank$(\Lambda)=1$ solved in Section \ref{sec:rank-1}, show that in all cases, for $\Lambda=\Lambda_{\ab,D}$, the Motzkin density $\md_{\ab}(D)=\md_{\mb{Z}^r/\Lambda}(B_\Lambda)$ is attained by a periodic $B_\Lambda$-avoiding set, thus giving a positive answer to Question \ref{Q:periodicity}, hence also to Questions \ref{Q:min} and \ref{Q:rationality}.
\end{proof}

\section{Bounds for the Motzkin problem in $\mb{T}$}\label{sec:Fourier}

In this section we discuss certain lower and upper bounds for Motzkin densities in the setting of the circle group $\mb{T}$. We begin with lower bounds.

For $t\in\T^r$ we shall use the following slight variant of the notation \eqref{eq:Lambda}:
\[
 \Lambda(t)=\big\{n\in\Z^r\,:\,n\cdot t= 0\in\T\big\},
\]
where $n\cdot t= n_1t_1+\cdots+n_rt_r\mod 1$. In other words $\Lambda(t)=\Lambda_{\mb{T},D}$ in the sense of \eqref{eq:Lambda}, for $D=\{t_i : i\in [r]\}$. We use $\|x\|$ to indicate the distance from $x\in\R$ to the nearest integer. This induces the natural metric on $\mb{T}$.

\begin{lemma}\label{lem:ext_kron}
 Let $\alpha,\beta\in\T^r$ such that $\Lambda(\alpha)\subset \Lambda(\beta)$.
 Then for each $\varepsilon>0$ there exists $N\in\Z$ such that $\|N\alpha_j-\beta_j\|<\varepsilon$ for every $j\in\{1,2,\dots, r\}$.
\end{lemma}

\begin{proof}
The lattice $\Lambda(\alpha)$ can be identified as the annihilator $\alpha^{\perp}:=\{\gamma\in \widehat{\mb{T}^r}:\gamma(\alpha)=0\}$. Then $\{\beta\in \mb{T}^r: \Lambda(\alpha)\subset \Lambda(\beta)\}$ is the closed (hence compact) subgroup $(\alpha^{\perp})^{\perp}=\overline{\mb{Z}\alpha}$ of $ \mb{T}^r$ (for this last equality, see for instance \cite[Lemma 2.1.3 ]{rudin}). Since the metric $d_\infty(x,y):=\max_{j\in [r]} \|x_j-y_j\|$ generates the topology on $\mb{T}^r$, the conclusion follows immediately.
\end{proof}

We use this to obtain the following lower bound on Motzkin densities in $\mb{T}$, whose proof also yields a construction of $D$-avoiding sets in $\mb{T}$ whose measures approximate arbitrarily closely the Motzkin density in many cases (as we shall see below).

\begin{theorem}\label{thm:LB}
Let $D=\{\alpha_1,\ldots,\alpha_r\}\subset\mb{T}\setminus\{0\}$, and let $\alpha=(\alpha_1,\ldots,\alpha_r)$. Then
\begin{equation}\label{eq:Bdef}
\md_\mb{T}(D)\;\geq\; 
 \kappa_\mb{T}(D):= \max\big\{\min_{i\in [r]}(\|\beta_i\|): \beta=(\beta_1,\ldots,\beta_r)\in \overline{\mb{Z}\alpha}\big\}.
\end{equation}
\end{theorem}

\begin{proof}
Let $\kappa= \kappa_\mb{T}(D)$. Note that $\kappa$ does not change if we restrict $\beta$ in \eqref{eq:Bdef} to having $\beta_i\neq 0$ for all $i\in [r]$. Fix any such $\beta\in \overline{\mathbb{Z}\alpha}$, and let $b=\min_i \|\beta_i\|>0$. Fix any $\varepsilon\in (0,b)$ and let $N$ be an integer given by Lemma \ref{lem:ext_kron}, so that $d_\infty(N\alpha,\beta)<\varepsilon$ (in particular, we must have $N\neq 0$).

We claim that the following set is $D$-avoiding:
\begin{equation}\label{eq:D-setconstr}
A=N^{-1}[-\tfrac{b-\varepsilon}{2},\tfrac{b-\varepsilon}{2}]:=\{x\in \mb{T}: \|Nx\|\leq \tfrac{b-\varepsilon}{2}\}=\bigcup_{k\in [0,|N|-1]}
  \Big(
  \tfrac{k}{N}+ \big[-\tfrac{b-\varepsilon}{2N},\tfrac{b-\varepsilon}{2N}\big]
  \Big).
 \end{equation}
Indeed, if we had $\alpha_j=a_1-a_2$ with $a_1,a_2\in A$, then we would have $\|N\alpha_j\|\leq b-\varepsilon$, so $\|N\alpha_j-\beta_j\|\geq \|\beta_j\|-\|N\alpha_j\|\geq \varepsilon$, a contradiction.

Given the claim, we deduce that $\md_{\mb{T}}(D)\geq \mu_{\mb{T}}(A)=b-\varepsilon$, and as $\varepsilon$ was arbitrary, we conclude that $\md_{\mb{T}}(D)\geq b$. Taking now the maximum over such $b$, the result follows.
\end{proof}
\begin{remark}
The quantity $\kappa_{\mb{T}}(D)$ is an analogue for $\mb{T}$ of a parameter that has been studied in the classical setting of Motzkin's problem in $\mb{Z}$ since the original paper \cite{C&G}, and is sometimes called the \emph{kappa value}; see \cite[Theorem 1]{C&G}, \cite[Remark 1]{Hara}, and \cite{L&R}.  Letting $q_d:\mb{R}^r\to\mb{T}^r$ be the canonical homomorphism $x\mapsto (x_i\!\mod 1)_{i\in [r]}$, we have $\overline{\mb{Z}\alpha}=(\alpha^{\perp})^{\perp}=q_r(\{\beta'\in \mb{R}^r:\forall\,n\in \Lambda(\alpha),\, n_1\beta'_1+\cdots+ n_r\beta'_r\in \mb{Z}\})$. In particular, this last set includes $q_r(\Lambda(\alpha)^*)$ where $\Lambda(\alpha)^*$ is the dual lattice of $\Lambda(\alpha)$ (i.e.\ the lattice consisting of $v\in \Span_{\mb{R}}(\Lambda(\alpha))$ satisfying $v_1x_1+\cdots+v_r x_r \in\mb{Z}$ for every $x\in \Lambda(\alpha)$). Thus $\kappa_{\mb{T}}(D)\geq \max\big\{\min_{i\in [d]}(\|\beta_i\|): \beta \in q_r(\Lambda(\alpha)^*)\big\}$, where the latter maximum can be easier to compute than $\kappa_{\mb{T}}(D)$ itself. 
\end{remark}
The following result is obtained in \cite{CCRS} using ergodic and graph theoretic tools. We give here a shorter proof using mainly Theorem \ref{thm:LB}.

\begin{theorem}\label{thm:half_parity}
Let $\alpha\in\mb{T}^r$ and suppose that $\Lambda(\alpha)$ has a set of generators $\{g_j\}_{j=1}^r$ such that the sum of the coordinates of each $g_j$ is even. Then $\md_{\T}(D)=1/2$ for $D=\{\alpha_1,\dots,\alpha_r\}$. Moreover, if $\alpha\not\in\Q^r$, then no $D$-avoiding set $A\subset\mb{T}$ satisfies $\mu(A)=1/2$.
\end{theorem}

\begin{proof}
Let $\beta_1=\beta_2=\dots=\beta_r=1/2$ and note that any $n\in\Lambda(\alpha)$ has $n_1+n_2+\dots +n_r$ even and hence $n\in \Lambda(\beta)$, so we can apply Theorem~\ref{thm:LB} with this $\beta$, deducing $\md_{\mb{T}}(D)=1/2$ as claimed. To see the last part, note that a $D$-avoiding set would in particular have to be $\{\alpha_j\}$-avoiding, for a value of $j$ such that $\alpha_j\not\in\mb{Q}$, and then such a set cannot have measure $1/2$, by the last sentence of \cite[Theorem 2.4]{CCRS}.
\end{proof}

\begin{remark}
The fact that the supremum $\md_{\mb{T}}(D)=1/2$ in Theorem \ref{thm:half_parity} is not attained for $\alpha\not\in\mb{Q}^r$ does not yield a negative answer to Question \ref{Q:min}. Indeed, for the corresponding number $\md_{\mb{Z}^r/\Lambda}(B_\Lambda)$ (equal to $\md_{\mb{T}}(D)$ by Theorem \ref{thm:trans}), we claim that the infimum in the expression \eqref{eq:mdinf} of this number \emph{is} attained. To prove this, note that the assumption that the coordinate sum of every $g_j$  is even implies that there is a homomorphism $q:\mb{Z}^r/\Lambda\to\mb{Z}/2\mb{Z}$, well defined by $q(x+\Lambda)= \sum_{i\in [r]} x_i\!\! \mod 2$. For $i=0,1$ let $V_i=q^{-1}(\{i\})$. Then, any set $S\subset \mb{Z}^r/\Lambda$ that is entirely included in $V_i$ (for $i=0$ or $1$) must be $B_\Lambda$-avoiding (as then $S-S\subset V_0$, whereas $B_\Lambda\subset V_1$). Therefore, to prove our claim it suffices to show that there is a finite set $F\subset \mb{Z}^r/\Lambda$ which tiles $\mb{Z}^r/\Lambda$ and such that $|V_0\cap F|=|V_1\cap F|$ (as then, with the notation from \eqref{eq:mdinf}, we have $\phi_{B_\lambda}(F)/|F|=|V_0\cap F|/|F|=1/2$). To see that such a tiling set $F$ exists, one can use \eqref{eq:FundThm}. For instance, if the subgroup $F'=\mb{Z}_{\alpha_1}\oplus\cdots\oplus \mb{Z}_{\alpha_d}\oplus\{0\}$ does not already satisfy $|V_0\cap F'|=|V_1\cap F'|$, then there must exist $v\in \{0\}\oplus\mb{Z}^{r-d}$ such that $|V_0\cap (v+F')|=|V_1\cap F'|$ (otherwise $V_0$, $V_1$ would not have equal density in $\mb{Z}^r/\Lambda$). But then the tiling set $F=F'\cup (v+F')$ satisfies $|V_0\cap F|=|V_1\cap F|$, and our claim follows.
\end{remark}
An immediate consequence of Theorem \ref{thm:half_parity} is the following corollary, obtained when $\{1\}\cup D$ is linearly independent over $\Q$, using that in this case $\Lambda(\alpha)=\{0\}$ (this consequence was proved differently in \cite[Theorem 2.4]{CCRS}).

\begin{corollary}[{}]\label{cor:KronFree}
If $\{1\}\cup D$ is linearly independent over $\Q$ then $\text{\rm Md}_{\mb{T}}(D)=1/2$ and no $D$-avoiding set $A\subset\mb{T}$ satisfies $\mu(A)=1/2$.
\end{corollary}

In the other extreme, some rational cases are also covered by Theorem \ref{thm:half_parity} without the use of graph theory made in \cite[\S 4]{CCRS}.

\begin{corollary}
For any set $D=\{a_1/q_1,\dots, a_r/q_r\}$ composed of irreducible fractions with $q_j\equiv 2\pmod{4}$, we have $\text{\rm Md}_{\T}(D)=1/2$. Moreover, in this case there exists a $D$-avoiding set $A\subset\T$ with $\mu(A)=1/2$.
\end{corollary}

\begin{proof}
 By the irreducibility, $a_j$ is odd and $2M/q_j$ is also odd with $M=\text{\rm lcm}(q_1/2,\dots, q_r/2)$.
 If $n\in \Lambda_{\mb{T},D}$, i.e.\ if $\sum n_ja_j/q_j=n\in\mb{Z}$, then multiplying by $2M$ to clear the denominators, we deduce that $\sum n_j a_j 2M/q_j$ is even, whence $\sum n_j$ is also even and so the assumption in Theorem \ref{thm:half_parity} is satisfied.

For the last part, note that if $N$ is an odd multiple of $M$ then we have the equality $\|Na_j/q_j-1/2\|=0$, so the construction in \eqref{eq:D-setconstr} works with $\varepsilon = 0$.
\end{proof}

It is natural to wonder whether the inequality in \eqref{eq:Bdef} is always an equality. For rank$(\Lambda_{\mb{T},D})=0$ it is easy to see that both sides of \eqref{eq:Bdef} are 1/2. Let us analyze the case rank$(\Lambda_{\mb{T},D})=1$. 

Looking first at $r=2$, this means that $D=\{t_1,t_2\}\subset\mb{T}\setminus\mb{Q}$ satisfying $n_1t_1+n_2t_2=n\in\Z$ with $n_j\in\Z$ and $\gcd(n_1,n_2,n)=1$. 
If $2\mid n_1+n_2$, then Theorem~\ref{thm:half_parity}  gives us 
$\md_{\T}(D)=1/2$. If $2\nmid n_1+n_2$, we have by \cite[Theorem 1.2]{CCRS} (or Theorem \ref{thm:rk1-intro} in this paper) $\md_{\mb{T}}(D)=\frac{|n_1|+|n_2|-1}{2(|n_1|+|n_2|)}$ and choosing $\beta=(\text{sgn}(n_1)\mu,\text{sgn}(n_2)\mu)$, we have that $n_1\beta_1+n_2\beta_2=(|n_1|+|n_2|-1)/2\in\Z$, whence $\beta\in((t_1,t_2)^\perp)^\perp$ and so $B\ge \md_{\mb{T}}(D)$, so we have equality in \eqref{eq:Bdef} in this case. 

A similar argument for $r>2$, using the solution for $\md_{\mb{T}}(D)$ in Theorem \ref{thm:rk1-intro}, shows that equality in \eqref{eq:Bdef} holds in the full case rank$(\Lambda_{\mb{T},D})=1$.

\begin{remark}\label{rem:Bnonsharp}
There are cases in which the inequality in \eqref{eq:Bdef} is strict. Indeed, if $d_1,\ldots,d_r$ are integers in $[N-1]$, each coprime with $N$, then letting $D=\{d_1/N,\ldots,d_r/N\}$ viewed as a subset of $\mb{T}$, we have $\md_{\mb{T}}(D)=\md_{\mb{Z}/N\mb{Z}}(\{d_1,\ldots,d_r\})$, and this last quantity in turn equals $\alpha(G)/N$ where $G$ is the circulant graph $\textrm{Cay}(\mb{Z}/N\mb{Z}, \{d_1,\ldots,d_r\})$ (see e.g.\ \cite[Lemma 2.10]{CCRS}). Moreover, in this case $\kappa_{\mb{T}}(D)$ equals $N$ times the quantity $\lambda(G)$ studied for circulant graphs in \cite{G&Z}. As established in \cite[Theorem 4]{G&Z}, if our circulant graph $G$ above is not \emph{star-extremal}, then $\lambda(G)<\alpha(G)$ and so $\kappa_{\mb{T}}(D)<\md_{\mb{T}}(D)$. At the end of the paper \cite{G&Z} a simple example of such a circulant graph is described, which in our setting translates into the following: for $D=\{1/13,\;3/13,\;4/13\}$ we have $\kappa_{\mb{T}}(D)=2/13<\md_{\mb{T}}(D)=3/13$ (as can be checked by computations).
\end{remark}

Concerning upper bounds, our first remark is that for general $r$-element sets $D\subset \mb{T}$, an approach to provide upper bounds for $\md_{\mb{T}}(D)$ is indicated by Theorem \ref{thm:trans} combined with \eqref{eq:mdinf}: for every finite tiling subset $F$ of the discrete group $\mb{Z}^r/\Lambda_{D,\mb{T}}$, we have
\begin{equation}\label{eq:UB1}
\md_{\mb{T}}(D)\leq \max_{A\subset F: (A-A)\cap B_\Lambda = \emptyset} \frac{|A|}{|F|},
\end{equation}
and we know that the infimum of the right side here over such sets $F$ \emph{equals} $\md_{\mb{T}}(D)$. Developing this approach further, to obtain more explicit expressions for the right side of \eqref{eq:UB1} in terms of $D$, seems to be an interesting direction for future work.

We close this section by presenting a different approach to obtain upper bounds for $\md_{\mb{T}}(D)$ in certain special cases, using Fourier analytic techniques. Fourier analysis has been applied in several works to produce upper bounds on Motzkin densities (or measures of intersectivity), mainly in finite abelian groups, as part of a general method going back to Delsarte (see for instance \cite{Delsarte,M&R,M&R2}). Such applications to Motzkin's problem have been fewer in continuous groups. The work \cite{B&R} develops the Delsarte method in locally compact abelian groups, but only for open sets $D$ (denoted $\Omega_+$ in \cite[Theorem 1.1]{B&R}), and in a way that does not adapt immediately to finite sets $D$ as in this paper. Here we shall illustrate the use of trigonometric polynomials to obtain upper bounds for $\md_{\mb{T}}(D)$ for some families of finite sets $D\subset \mb{T}$, and we will also discuss some limitations of this approach.

As usual, we write $e(x)$ to abbreviate the complex exponential $e^{2\pi i x}$, the simplest nontrivial character on $\mb{T}$. We call \emph{real cosine polynomial} any function on $\mb{T}$ of the form
\begin{equation}\label{cos_pol}
 C(t)
 =
 \sum_{k=0}^r
 c_k\cos( 2\pi kt)
 \qquad\text{with}\quad c_k\in\mb{R}.
\end{equation}
We define the \emph{support} of its Fourier coefficients as $\{k\in \mb{Z}_{\geq 0}\,:\,c_k\ne 0\}$.

We consider a non-trivial family of sets $D\subset \mb{T}$ within the simplest case in which we currently do not have an exact formula for $\md_{\mb{T}}(D)$, namely the case rank$(\Lambda_{\mb{T},D})=r-1$. The family in question is that of sets of the form $D=\{k_1x,k_2x,\ldots,k_rx\}$ for positive integers $k_i$ and $x\in \mb{T}$. In this setting, a strong form of the $D$-avoiding property can be related to an optimization problem for cosine polynomials, resulting in upper bounds for $\md_{\mb{T}}(D)$.

\begin{proposition}\label{prop:pos_cos}
Let $x\in \mb{T}$ be arbitrary, and let $D=\{k_1x,k_2x,\dots, k_rx\}$ where $k_j\in\mb{Z}_{>0}$. Let $C(t)$ be of the form \eqref{cos_pol} with $c_0=1$ and its Fourier coefficients supported on the set $\mathcal{K}=\{0,k_1,\dots, k_r\}$. If $C\ge 0$, then any Borel set $A\subset\T$ with $\mu(A)> 1/C(0)$ satisfies $(A-A)\cap D\ne 0$, i.e.\ $\md_{\mb{T}}(D)\leq 1/C(0)$. Moreover, this also holds for $\mu(A)= 1/C(0)$ if $x$ is irrational. 
\end{proposition}

\begin{proof}
By standard Fourier-analytic results, for any Borel sets $A,B$ in $\mb{T}$ and every $x\in \mb{T}$, we have
\begin{equation}\label{eq:FourierExp}
\chi_A*\chi_B(t)=
  \mu(A)\,\mu(B)
  + \sum_{n\in\mb{Z}\setminus\{0\}}
\wh{\chi_A}(n)\, \wh{\chi_B}(n)\,e^{2\pi i n t}.
 \end{equation}
 Consider the function
 $
  F(t ) = \sum_{j=1}^dc_{k_j} \chi_A*\chi_{-A}(k_j t)
 $.
 By \eqref{eq:FourierExp},
 \[
  F(t)=
  \sum_{j=1}^d c_{k_j}\mu(A)^2
  +
  2\sum_{n=1}^\infty
  |\wh{\chi_A}(n)|^2
  \big(C(nt)-1\big).
 \]
 Adding \eqref{eq:FourierExp} with $t=0$ (Parseval's identity), we obtain
 \[
  F(t)
  =
  C(0)\mu(A)^2-\mu(A)
  +
  2\sum_{n=1}^\infty
  |\wh{\chi_A}(n)|^2
  C(nt).
 \]
 If $\mu(A)>1/C(0)$, then $F(x)>0$, so $\chi_A*\chi_{-A}(k_j x)\ne0$ for some $j$, hence $k_jx\in A-A$.

 If $x\not\in\Q$ then $nx$ are distinct modulo~$1$ for $n\in\mb{Z}_{>0}$. As the number of roots of $C(t)=0$ in $\T$ is finite, we have $C(nx)>0$ for $n$ large and $\mu(A)=1/C(0)$ gives again $F(x)>0$.
\end{proof}

\begin{remark}
The Motzkin density $\md_{\T}(D_x:=\{k_1x,\ldots,k_rx\})$ considered in Proposition \ref{prop:pos_cos} can vary as $x$ varies in $\mb{T}$ with fixed integers $k_i$. For instance, supposing that the $k_i$ are coprime, then for any irrational $x\in \mb{T}$ we have $\md_{\T}(D_x)=\md_{\mb{Z}}(D':=\{k_1,\ldots,k_r\})$ (indeed in this case $\Lambda$ has rank $r-1$ and has a primitive basis, and it can be checked that the isomorphism $\mb{Z}^r/\Lambda\to\mb{Z}$ in the proof of Proposition \ref{prop:equivZkrefined} takes $B_\Lambda$ to $D'$). On the other hand, for $x$ rational we can have $\md_{\T}(D_x)\neq \md_{\mb{Z}}(D')$. For example, for $D'=\{1,3,4\}$, by \cite[Theorem 4(b)]{Hara} we have $ \md_{\T}(D_x)= \md_{\mb{Z}}(D')=2/7$ for $x\in \mb{T}\setminus\mb{Q}$, while for $x=1/13$ we have $\md_{\T}(D_x)=3/13$, as mentioned in Remark \ref{rem:Bnonsharp}.
\end{remark}

As a first illustration of the bound given by Proposition \ref{prop:pos_cos}, we can give a simple family of non-trivial examples where this bound is sharp, namely the sets $D=\{x,2x,\dots, Nx\}$, $N\in \mb{N}$. For this we shall use the following fact concerning the Fej\'er kernel 
\[
F_N(t)=\sum_{m=-(N-1)}^{N-1}\Big(1-\frac{|m|}{N}\Big)e(mt).
\]
\begin{proposition}\label{prop:fej_opt}
 Let $C$ be a real cosine polynomial as in \eqref{cos_pol} with $c_0=1$.
 If $C\ge 0$ then $C(0)\le N+1$. The equality is reached only for the Fej\'er kernel~$F_{N+1}$.
\end{proposition}

\begin{proof}
The Fej\'er-Riesz theorem \cite[\S6.3]{montgomery} assures that every nonnegative trigo\-nometric polynomial is of the form
\[
C(t)=\Big|\sum_{k=0}^N b_k e(kt)\Big|^2.
\]
For $C$ even, we can take $b_k$ to be real \cite[\S3]{GoMo}. The condition $1=c_0=\int_0^1C$ is equivalent to $\sum b_k^2=1$ and Cauchy-Schwarz inequality gives $C(0)\le (N+1)\sum b_k^2=N+1$. Equality holds here only when $b_0=b_1=\dots=b_N$, and the normalization then imposes $b_j^2= 1/(N+1)$. Then $C(t)=\frac{1}{N+1}\big|\sum_{k=0}^N e(kt)\big|^2=F_{N+1}(t)$.
\end{proof}

Taking $C=F_{N+1}$ in Proposition~\ref{prop:pos_cos}, we obtain that $\md_{\mb{T}}(D)\leq 1/(N+1)$, and this can be seen to be sharp. For instance, for real and arbitrarily small $x>0$, and $M=\lfloor \tfrac{1-x}{x(N+1)}\rfloor$, the set $A=\bigcup_{k=0}^M \big[k(N+1)x,k(N+2)x\big)\subset [0,1)$  is $D$-avoiding of measure arbitrarily close to $ 1/(N+1)$. Let us note, however, that the upper bound $\md_{\mb{T}}(D)\leq 1/(N+1)$ admits an elementary simple proof, without using Fourier analysis, because the set $\{1,2,\dots,N\}$ is closed under taking absolute values of the differences. Namely, if $A$ is $D$-avoiding, then for any $0\le \ell<k\le N$ we have $(k-\ell)x\not\in A$, and so $(A+kx)\cap(A+\ell x)=\emptyset$, which implies that $1\ge\mu(\bigcup_{k=0}^N k+A)= \sum_{k=0}^N\mu(k+A)=(N+1)\mu(A)$. On the other hand, the following extreme case afforded by the last sentence of Proposition~\ref{prop:pos_cos} does not seem to be covered by a variation on this elementary argument.

\begin{corollary}
 Let $D=\{x,2x,\dots, Nx\}$ with $x\in\T\setminus\Q$. There does not exist a $D$-avoiding subset $A\subset\T$ with $\mu(A)=1/(N+1)$.
\end{corollary}

Providing upper bounds via Proposition \ref{prop:pos_cos} for other choices of sets $D=\{k_ix:i\in[r]\}$ relies on designing appropriate cosine polynomials, and there are certainly cases in which this can fail to produce sharp bounds. 
Consider for instance the case $D=\{x,3x,8x\}$ for an arbitrary fixed $x\in\T$. We claim that
\begin{equation}\label{eq:cosC}
C(t):= \frac{9}{28}\cos(8t) + \frac{209}{252}\cos(3t)+ \frac{1553}{6048}\cos(t) +1 >0.
\end{equation}
This can be proved applying Sturm's theorem to 
$\frac{9}{28}T_8(x) + \frac{209}{252}T_3(x)+ \frac{1553}{6048}T_1(x) +1$
where $T_k$ are the Chebyshev polynomials defined by $\cos(k\arccos x)$ for $x\in [-1,1]$. From the above positivity, it follows by Proposition~\ref{prop:pos_cos} that $\md_{\mb{T}}(D)\le 6048/14561\approx 0.415356$. 

On the other hand, it can be shown (by relating $\md_{\mb{T}}(D)$ to $\md_{\mb{Z}}(\{1,3,8\})$ and using \cite[Theorem 4(b)]{Hara}) that for appropriate $x$ we have $\md_{\mb{T}}(D)=4/11\approx 0.3636$. Unfortunately, we cannot reach this number with variations of the polynomial $C$ above. This can be proved by considering the linear programming problem
$x_1\cos(8t_k) + x_2\cos(3t_k)+ x_3\cos(t_k) +1 > 0$ where $t_k=k\pi/N$, $k=0,\dots , N-1$, with the objective function 
$x_1+x_2+x_3+1$. For $N$ large enough, a computer calculation shows that the objective function is less than $2.4088$, in particular, no positive cosine polynomial can give $\md_{\mb{T}}(D)<1/2.4088\approx 0.415144$, and so the polynomial $C$ in \eqref{eq:cosC} is nearly optimal. 

\appendix

\section{Reduction to metrizable compact abelian groups}\label{App:reduc}

We prove the following claim made in the first part of the proof of Theorem  \ref{thm:trans}.
\begin{proposition}\label{prop:metrizclaim}
Let $\ab$ be a compact abelian group and let $D$ be a finite subset of $\ab\setminus\{0\}$. Then there exists a compact metrizable abelian group $H$ and a  continuous surjective homomorphism $q:\ab\to H$, such that $\md_{\ab}(D)=\md_H(q(D))$ and $\Lambda_{H,q(D)}=\Lambda_{\ab,D}$.
\end{proposition}
\begin{proof}
We obtain the map $q:\ab\to H$ using Pontryagin duality: we shall find a countable subgroup $\Gamma$ of the dual group $\wh{\ab}$ such that $H:=\wh{\Gamma}$ satisfies $\md_{\ab}(D)=\md_H(q(D))$ and $\Lambda_{H,q(D)}=\Lambda_{\ab,D}$. Since $\Gamma$ is countable we will also have that $H$ is metrizable.

First, we let $S_0$ be a countable subset of $\wh{\ab}$ consisting of characters which detect if $n\in \mb{Z}^r\setminus \Lambda_{\ab,D}$. More precisely, for every $n\in \mb{Z}^r\setminus \Lambda_{\ab,D}$, since $n_1t_1+\cdots+n_r t_r \neq 0$, the fact that $\wh{\ab}$ separates points on $\ab$ (see e.g.\ \cite[p.\ 24]{rudin}) implies that there exists $\chi_n\in \wh{\ab}$ such that $\chi_n(n_1t_1+\cdots+n_r t_r)\neq 1$. We let $S_0=\{\chi_n:n\in \mb{Z}^r\setminus \Lambda_{\ab,D}\}$. Note that for any countable subgroup $\Gamma$ of $\wh{\ab}$, if $\Gamma\supset S_0$ then for the corresponding quotient map $q:\ab \to \wh{\Gamma}$ we have that $\wh{\Gamma}$ is metrizable and $\Lambda_{\wh{\Gamma},q(D)}=\Lambda_{\ab,D}$. Indeed the inclusion $\Lambda_{\wh{\Gamma},q(D)}\supset \Lambda_{\ab,D}$ is clear, and for the opposite inclusion, since $\Gamma\supset S_0$ we have $\mb{Z}^r\setminus \Lambda_{\ab,D} \subset \mb{Z}^r\setminus \Lambda_{\wh{\Gamma},q(D)}$. To see the latter inclusion note that if $n_1t_1+\cdots n_r t_r\neq 0$ then by \cite[p.\ 35, (1)]{rudin} there is a character $\phi$ on $\ab/\wh{\Gamma}$ such that $\phi(q(n_1t_1+\cdots n_r t_r))=\chi_n(n_1t_1+\cdots n_r t_r)\neq 1$, so $q(n_1t_1+\cdots n_r t_r)\neq 0$.

Now we shall add countably many characters to $S_0$ to ensure also that the property involving Motzkin densities holds. For each positive integer $k$, let $A_k$ be a $D$-avoiding Borel subset of $\ab$ such that $\mu_{\ab}(A_k)\geq \md_{\ab}(D)-1/k$. By Plancherel's theorem, the Fourier transform $\wh{1}_{A_k}$ is supported on a countable subset $S_k\subset \wh{\ab}$. Let $S=S_0\cup \bigcup_{k\in \mb{N}}S_k$. 

As $S$ is countable, the subgroup $\Gamma=\langle S\rangle$ is countable. Now let  $H=\wh{\Gamma}$, which is a metrizable quotient of $\ab$. Letting $q:\ab\to H$, it remains to show that $\md_{\ab}(D)=\md_H(q(D))$. Note that if $A\subset H$ is $q(D)$-avoiding, then $q^{-1}(A)$ is $D$-avoiding with $\mu_{\ab}(q^{-1}(A))=\mu_H(A)$, so $\md_H(q(D))\leq \md_{\ab}(D)$. 
To see the opposite inequality, fix any $k\in \mb{N}$ and note that, since the support of $\wh{1_{A_k}}$ is included in $\Gamma$, there exists a Borel set $A_k'\subset H$ such that $\mu_{\ab}((q^{-1}A_k')\Delta A_k)=0$. This implies that $q^{-1}(A_k')$ is $D$-avoiding modulo a $\mu_{\ab}$-null set, and it follows that $A_k'$ is $q(D)$-avoiding modulo a $\mu_H$-null set. Removing this null set from $A_k'$ we obtain a $q(D)$-avoiding Borel set $A_k''$, with $\mu_H(A_k'')=\mu_{\ab}(q^{-1}A_k')=\mu_{\ab}(A_k)\geq \md_{\ab}(D)-1/k$, so $\md_H(q(D))\geq \md_{\ab}(D)-1/k$, whence the desired inequality follows by letting $k\to \infty$.
\end{proof}

\end{document}